\documentclass[12pt]{amsart}
\usepackage[cp1251]{inputenc}
\usepackage[english]{babel}
\usepackage{amsmath,amsthm,amssymb}
\newtheorem{theorem}{Theorem}[section]
\newtheorem{lemma}[theorem]{Lemma}
\newtheorem{corollary}[theorem]{Corollary}
\newtheorem{proposition}[theorem]{Proposition}
\theoremstyle{definition}

\newtheorem{example}[theorem]{Example}
\theoremstyle{remark}
\newtheorem{remark}[theorem]{Remark}

\numberwithin{equation}{section}

\begin{document}

\title[Fractional smoothness of distributions of polynomials]
{Fractional smoothness of distributions of polynomials
and a fractional analog of the Hardy--Landau--Littlewood inequality}

\author[Vladimir I. Bogachev et al.]{Vladimir~I.~Bogachev}
\address{National Research University Higher School of Economics, Moscow, Russia}
\curraddr{}
\email{vibogach@mail.ru}
\thanks{This work has been supported by the Russian Science Foundation Grant 14-11-00196 at
Lomonosov Moscow State University.}

\author[]{Egor~D.~Kosov}
\address{Faculty of Mechanics and Mathematics,
Moscow State University, Moscow, 119991 Russia}
\curraddr{}
\email{ked\_2006@mail.ru}
\thanks{}

\author[]{Georgii~I.~Zelenov}
\address{Faculty of Mechanics and Mathematics,
Moscow State University, Moscow, 119991 Russia}
\curraddr{}
\email{zelenovyur@gmail.com}
\thanks{}

\date{}

\maketitle

\begin{abstract}
We prove that the distribution density of any  non-constant polynomial
$f(\xi_1,\xi_2,\ldots)$
of degree~$d$ in independent standard Gaussian random variables~$\xi$
(possibly, in infinitely many variables)
always belongs to the
Nikol'skii--Besov  space~$B^{1/d}(\mathbb{R}^1)$
of fractional order~$1/d$ (and this order is best possible),
and an analogous result holds
for polynomial mappings with values in~$\mathbb{R}^k$.

Our second main result is
an upper bound on the total variation distance
between two probability measures on $\mathbb{R}^k$ via
the Kantorovich distance between them and
a suitable Nikol'skii--Besov  norm of their difference.

As an application we consider the total variation distance between the distributions
of two random $k$-dimensional vectors composed of
polynomials of degree $d$ in Gaussian random variables and show that
this distance is estimated
by a fractional power of the Kantorovich distance with an exponent depending only
on $d$ and $k$, but not on the number of variables of the considered polynomials.
\end{abstract}



\section{Introduction}\label{intro}

This paper is concerned with distributions of polynomials in Gaussian random variables
and estimates in the total variation distance between measures with densities from fractional
Nikol'skii--Besov  classes.

Our first main  result (presented in Section 4 and Section 5)
states that the distribution of any non-constant polynomial
of degree~$d$
(possibly, in infinitely many variables) with respect to a Gaus\-sian measure always belongs to the
Nikol'skii--Besov  space~$B^{1/d}(\mathbb{R}^1)$
(so that the order of smoothness depends only on the degree
of this polynomial and this order is best possible) and that
an analogous result holds for multidimensional
polynomial mappings. It is well-known that a non-constant polynomial in Gaussian
random variables has a distribution density, however, in many cases this density
is not locally bounded (which happens already for the square
of the standard Gaussian random variable), hence does not belong to an
integer order Sobolev class.
The established fractional regularity is the first general result in this direction.

Our second main result gives
new lower bounds for the Kantorovich distance
(all definitions are given in Section~2) between
probability measures on~$\mathbb{R}^k$; these bounds can be also viewed as upper bounds
for the total variation distance.
Our principal new result is a fractional
multidimensional analog of the classical Hardy--Landau--Littlewood inequality.
We obtain an upper bound on the total variation
distance between two probability measures on $\mathbb{R}^k$ in terms of
the Kantorovich distance between them and a suitable Nikol'skii--Besov  norm of their difference.
A~particular case of our inequality is the estimate of the total variation norm via
the Kantorovich  norm and the BV-norm established in~\cite{BSh},~\cite{BWSh}.
The classical Hardy--Landau--Littlewood result \cite{HLL} states that
$$
\|f'\|_1^2 \le C \|f\|_1\|f''\|_1
$$
for every integrable function $f$ on the real line with two integrable derivatives.
A~multidimensional analog of this bound was obtained in
\cite{BSh}, \cite{BWSh} (see also \cite{Kohn} and~\cite{Seis}) in the following form:
for every $k$, there is a number $C(k)$ such that for every two probability measures
$\mu$ and $\nu$ on $\mathbb{R}^k$ with densities $\varrho_\mu$ and $\varrho_\nu$
 belonging to the class $BV$ of functions of bounded variation one has
\begin{equation}\label{mHLL}
d_{{\rm TV}} (\mu,\nu)^2\le C(k) d_{{\rm K}} (\mu,\nu) \|D\varrho_\mu-D\varrho_\nu\|_{{\rm TV}} ,
\end{equation}
where $d_{{\rm TV}}$ is the total variation distance and
$d_{{\rm K}}$ is the Kantorovich distance (see definitions below).
In the one-dimensional case, this inequality is equivalent to the
Hardy--Landau--Littlewood inequality (and can be obtained from the latter
by passing to smooth compactly supported functions and taking for $f$
the difference of the distribution functions of the given measures).
However, this result does not  directly  apply to
polynomial images of Gaussian measures, our second main object.
For example, the distribution density of the $\chi^2$-distribution with one degree of freedom
is unbounded (it behaves like $t^{-1/2}$ near zero)
and does not belong to the class~$BV$.
For this reason, having in mind applications to distributions of polynomials (treated in Section~4 and~5),
in Section~3
we first obtain a suitable extension of (\ref{mHLL})
that involves fractional derivatives in place of gradients.
Namely, given two  Borel probability measures
$\nu, \sigma$ in the Nikol'skii--Besov  class $B^\alpha(\mathbb{R}^k)$,
$\alpha\in (0,1]$, we prove
 that
$$
\|\sigma-\nu\|_{{\rm TV}} \le
C(k,\alpha)\|\sigma - \nu\|_{B^\alpha}^{1/(1+\alpha)}d_{{\rm K}}
(\sigma, \nu)^{\alpha/(1+\alpha)}.
$$

As an application (considered in Sections~4 and~5)
we give upper bounds on the total variation distance
via the Kantorovich distance between the distributions
of two random $k$-dimensional vectors  whose  components are
polynomials of degree $d$ in Gaussian variables.
The former distance is estimated
by a certain fractional  power of the latter with an exponent depending only
on the degree $d$ and dimension $k$ of the vectors, but not on the number of variables
of these polynomials, which yields an immediate infinite-dimensional extension.
Our bounds improve the recent results of Nourdin, Nualart and Poly~\cite{NNP}.
This improvement is due to a new method based on the aforementioned fractional multidimensional
analog of the Hardy--Landau--Littlewood inequality and also involves
Nikol'skii--Besov  classes.
In this relation recall that
Nourdin and Poly \cite[Theorem 3.1]{NP}  proved the following interesting fact
(the concepts involved in the formulation are defined in the next section).
If $\{f_n\}$ is a sequence of polynomials of degree~$d$
on a space with a Gaussian measure~$\gamma$
such that their distributions $\gamma\circ f_n^{-1}$ converge weakly to an absolutely
continuous measure, then there is a number $C$ such that
$$
d_{{\rm TV}}(\gamma\circ f_n^{-1}, \gamma\circ f_m^{-1})\le
C d_{{\rm KR}}(\gamma\circ f_n^{-1}, \gamma\circ f_m^{-1})^{\theta}, \quad
\theta=\frac{1}{2d+1},
$$
where $d_{{\rm KR}}$ is the Kantorovich--Rubinstein distance
(see below; the term ``Fortet--Mourier distance'' used in  \cite{NP} is reserved
in our paper for the equivalent metric $d_{{\rm FM}}$ from the original paper~\cite{FM}).
The proof in \cite{NP} implies that,
for any two $\gamma$-measurable polynomials of degree $d$ with
variances $\sigma_f,\sigma_g$ in a given interval $(a,b)$ with $a>0$,
there is a number $C=C(a,b,d)$, depending only on $a,b,d$,  such that
$$
d_{{\rm TV}}(\gamma\circ f^{-1}, \gamma\circ g^{-1})\le
C d_{{\rm KR}}(\gamma\circ f^{-1}, \gamma\circ g^{-1})^{\frac{1}{2d+1}}.
$$
In the multidimensional case, it was shown in~\cite{NNP} that,
given a sequence of $k$-dimensional random vectors $f_n$ composed of $\gamma$-measurable polynomials
of degree $d$ such that their distributions $\gamma\circ f_n^{-1}$ converge weakly
and the expectations of the determinants of their Malliavin matrices are separated from zero,
for every
$$
\theta<\frac{1}{(k+1)(4k(d-1)+3)+1}
$$
there exists a number $C$ such that
$$
d_{{\rm TV}}(\gamma\circ f_n^{-1}, \gamma\circ f_m^{-1})\le
C d_{{\rm KR}}(\gamma\circ f_n^{-1}, \gamma\circ f_m^{-1})^{\theta}.
$$
Here we develop a different approach based on multidimensional analogs of the
Hardy--Landau--Littlewood inequality and in Section~4
we prove an estimate with a much better rate of convergence: given
$d\in\mathbb{N}$, $a,b>0$, for each positive number
$$
\theta<\frac{1}{4k(d-1)+1},
$$
there exists a number $C=C(d,a,b,\theta)$ such that, whenever $f$ and $g$ are
$k$-dimen\-sio\-nal polynomial mappings  of degree~$d$
(in an arbitrary, possibly, infinite, number of variables)
with variances of components bounded by $b$
and the expectations of the determinants of the Malliavin matrices separated from zero by~$a$,  one has
$$
d_{{\rm TV}}(\gamma\circ f^{-1}, \gamma\circ g^{-1})\le
C d_{{\rm KR}}(\gamma\circ f^{-1}, \gamma\circ g^{-1})^{\theta}.
$$

In Section 5 we consider separately the one-dimensional case
and also improve the aforementioned bound from \cite{NP} from
the power $\theta =(2d+1)^{-1}$
to nearly $(2d-1)^{-1}$, more precisely,
we establish the foregoing bound with any power $\theta< 1/(2d-1)$.
Moreover, with a worse constant we obtain a bound with the power
$\theta=1/(d+1)$, which is close to $1/d$ and the latter cannot be increased.
Finally, in Section~6 we give two related estimates connected with results from \cite{DavM} and~\cite{NP}.
The readers not interested in the infinite-dimensional case can just ignore
the corresponding statements; the essence of the paper is in finite-dimensional
results independent of the number of variables.
We thank I.~Nourdin for useful discussions.

\section{Definitions and notation}

The standard Gaussian measure $\gamma_n$ on $\mathbb{R}^n$ has density
$$
(2\pi)^{-n/2}\exp(-|x|^2/2).
$$
The image of a measure $\mu$ on a measurable space under a measurable mapping $f$
with values in $\mathbb{R}^k$ is denoted by the symbol $\mu\circ f^{-1}$
and defined by the formula
$$
\mu\circ f^{-1}(B)=\mu(f^{-1}(B))
\quad
\hbox{for every Borel set $B\subset \mathbb{R}^k$.}
$$
If $\xi_1,\ldots,\xi_n$ are independent standard Gaussian random variables,
$f\colon\, \mathbb{R}^n\to\mathbb{R}^k$,
then the law of $f(\xi_1,\ldots,\xi_n)$ is exactly~$\gamma_n\circ f^{-1}$.
If $k=1$, then the distribution density of $\mu\circ f^{-1}$ (if exists)
is the derivative of the function $t\mapsto \mu(f<t)$.

We set $\|\varphi\|_{\infty}=\sup_x |\varphi(x)|$ for any bounded function~$\varphi$
on any set.

The total variation distance $d_{{\rm TV}} (\mu,\nu)$ between two
Borel measures $\mu,\nu$ on $\mathbb{R}^k$ is generated by the norm
$$
\|\sigma\|_{{\rm TV}}  := \sup\biggl\{\int \varphi \, d\sigma, \ \varphi\in C_b^\infty(\mathbb{R}^k),
\ \|\varphi\|_\infty \le1 \biggr\}.
$$
The Kantorovich distance
(or the Kantorovich--Rubinstein distance \cite{K42}, \cite{KR},
sometimes erroneously called the Wasserstein distance)
between two Borel probability measures
$\mu,\nu$ on $\mathbb{R}^k$ with finite first moments is defined by the formula
$$
d_{{\rm K}} (\mu,\nu) := \sup\Bigl\{\int \varphi \,
d(\mu-\nu), \ \varphi\in C_b^\infty(\mathbb{R}^k),\ \|\nabla\varphi\|_\infty \le1\Bigr\}.
$$
For measures without moments, the following Fortet--Mourier distance
 can be used (see \cite[p.~277--279]{FM}; other distances including $d_{{\rm K}}$ are considered there):
$$
d_{{\rm FM}}(\mu,\nu) := \sup\biggl\{\int \varphi\, d(\mu-\nu), \
\varphi\in C_b^\infty(\mathbb{R}^k),\ \|\varphi\|_\infty+\|\nabla\varphi\|_\infty\le1\biggr\}.
$$
An equivalent distance (also called the Kantorovich--Rubinstein distance, since
it is a special case of a metric used in \cite[Theorem~1']{KR})
which is generated by equivalent norm is defined by
$$
d_{{\rm KR}}(\mu,\nu) := \sup\biggl\{\int \varphi\, d(\mu-\nu), \
\varphi\in C_b^\infty(\mathbb{R}^k),\ \|\varphi\|_\infty \le 1,\ \|\nabla\varphi\|_\infty\le1\biggr\}.
$$
These distances can be defined on general metric spaces where in place of $C_b^\infty$ one takes
the class of all bounded Lipschitz functions.
It is clear that $d_{{\rm KR}}\le d_{{\rm K}}$.

Recall  (see \cite{BIN},~\cite{Nikol77})
that the Nikol'skii--Besov class
$B^{\alpha}_{1,\infty}(\mathbb{R}^k)$ of order $\alpha\in (0,1)$
consists  of all functions
$\varrho\in L^1(\mathbb{R}^k)$ such that
$$
\|\varrho (\cdot+h)-\varrho\|_{L^1}\le C(\varrho) |h|^\alpha
 \quad \forall\, h\in \mathbb{R}^k
$$
for some number $C(\varrho)$;
it is also denoted by
$H^{\alpha}_1(\mathbb{R}^k)$ in~\cite{Nikol77},
by $B^{\alpha;1,\infty}(\mathbb{R}^k)$ in \cite{AF}
and by $\Lambda^{1,\infty}_\alpha$ in~\cite{Stein}.
This class is a particular case of the class
$H^{\alpha}_{p}(\mathbb{R}^k)$ defined similarly with the $L^p$-norm in place
of the $L^1$-norm.
Throughout we use the shortened notation $B^\alpha(\mathbb{R}^k)$.
Moreover, we use the symbol $B^1(\mathbb{R}^k)$ also for $\alpha=1$, which
corresponds to the class $BV(\mathbb{R}^d)$ of functions  of bounded variation
(which is smaller than the usual Nikol'skii--Besov class with $\alpha=1$ defined
via symmetric differences $\varrho (\cdot+h)+\varrho (\cdot-h) -2\varrho$).
However, it will be more convenient to deal with measures possessing
densities from these classic spaces rather than with functions.

Let $\nu$ be a bounded Borel measure on $\mathbb{R}^k$ and
let $\nu_h$ denote its shift by the vector~$h$:
$$
\nu_h(A)=\nu(A-h).
$$
Let $0<\alpha\le 1$. Then the class  $B^\alpha(\mathbb{R}^k)$ coincides with
the class of densities of bounded Borel measures $\nu$ on $\mathbb{R}^k$ such
that, for some number $C_\nu$,  one has
$$
\|\nu_h - \nu\|_{{\rm TV}} \le C_\nu|h|^\alpha
\quad \forall\, h\in \mathbb{R}^k.
$$
We shall identify measures with their densities and speak of measures
in the class $B^\alpha(\mathbb{R}^k)$ in this sense.

We need the following norm on the space $B^\alpha(\mathbb{R}^k)$:
$$
\|\nu\|_{B^\alpha} := \inf\{C\colon\ \|\nu - \nu_h\|_{{\rm TV}} \le C|h|^\alpha\}.
$$
It is readily seen that this is indeed a norm. However, the space $B^\alpha(\mathbb{R}^k)$ is not
complete with this norm: its standard Banach norm is given by
$
\|\nu\|_{{\rm TV}}+ \|\nu\|_{B^\alpha}.
$
The latter is larger than $\|\nu\|_{B^\alpha}$ and the two norms are not equivalent:
indeed, letting $f_n(x)=1$ on $[-n,n]$, $f_n(0)=0$ outside $[-n-1,n+1]$ and
$f_n(x)=n+1-|x|$ if $n<|x|<n+1$, we have $\|f_n\|_{L^1}\to \infty$,
$\sup_n \|f_n\|_{B^\alpha}<\infty$, where we identify $f_n$ with the measure $f_ndx$.
The situation is similar with Sobolev spaces once we use only the norm of the gradient.

The following embedding holds (see \cite[Section~6.3]{Nikol77}):
\begin{equation}\label{nikolski-emb}
B^{\alpha}(\mathbb{R}^k)\subset H^{\beta}_{p}(\mathbb{R}^k)
\subset L^p(\mathbb{R}^k),
\quad
 \beta=\kappa \alpha, \
\kappa =1-\frac{k(p-1)}{\alpha p}.
\end{equation}
Hence all measures from
$B^\alpha(\mathbb{R}^k)$ have densities in $L^p(\mathbb{R}^k)$ for all $p<k/(k-\alpha)$.
These embeddings to $L^p$ on balls (compositions with restrictions) are compact.

For infinite-dimensional extensions of our results we recall the corresponding concepts.
A  probability measure defined on the Borel $\sigma$-field of
a locally convex space $X$ is called Radon
if its value on each Borel set is the supremum of measures of compact subsets of this set.
A centered Radon Gaussian measure $\gamma$
is a Radon probability measure on~$X$ such that
every continuous linear functional $f$ on $X$ is a centered Gaussian random variable
on $(X,\gamma)$; in other words, $\gamma\circ f^{-1}$ is either Dirac's measure at
zero
or has a distribution density $(2\pi \sigma_f)^{-1/2}\exp(-x^2/(2\sigma_f))$,
where $\sigma_f=\|f\|_{L^2(\gamma)}^2$.
On complete separable metric spaces all Borel measures are automatically Radon.
Typical examples of Gaussian measures are the countable power of the standard Gaussian measure
on~$\mathbb{R}$ (defined on the countable power $\mathbb{R}^\infty$
of~$\mathbb{R}$) and the Wiener
measure (see \cite{GM} and \cite{B14} about Gaussian measures).

Let $H\subset X$ be the Cameron--Martin space of the measure $\gamma$, i.e.,
the space of all vectors $h$ such that $\gamma_h \sim \gamma$.
If $\gamma$ is the countable power of the standard Gaussian measure on the real line,
then $H$ is the usual Hilbert space~$l^2$
(of course, for the standard Gaussian measure on $\mathbb{R}^d$
the Cameron--Martin space is $\mathbb{R}^d$ itself). The Cameron--Martin space of the Wiener
measure on $C[0,1]$ is the space of absolutely continuous functions on $[0,1]$
vanishing at $0$ and having derivatives in $L^2[0,1]$.
For a general Radon Gaussian measure the Cameron--Martin space
is also a separable Hilbert space (see \cite[Theorem 3.2.7 and Proposition 2.4.6]{GM})
with the inner product $\langle\cdot,\cdot\rangle_H$ and the norm $|\cdot|_H$ defined
by
$$
|h|_H=\sup \biggl\{ l(h)\colon\, \int_X l^2\, d\gamma \le 1, \ l\in X^{*}\biggr\}.
$$

Let $\mathcal{P}^d(\gamma)$ be the closure in $L^2(\gamma)$ of the linear space
of all  functions of the form
$$
\varphi _d(l_1(x),\ldots,l_m(x)),
$$
where
$\varphi _d(t_1,\ldots,t_m)$ is a polynomial in $m$ variables of degree $d$
and $l_1,\ldots,l_m$ are continuous linear functionals on $X$
($m$~can be an arbitrary natural number).
Functions from the class $\mathcal{P}^d(\gamma)$ will be called measurable polynomials of degree~$d$.

The Wiener chaos $\mathcal{H}_d$ of order $d$ is defined as the orthogonal complement of
$\mathcal{P}^{d-1}(\gamma)$ in $\mathcal{P}^d(\gamma)$,
 $\mathcal{H}_0$ is the space of constants.
It is well-known (see, e.g.,  \cite[Section~2.9]{GM})
that $L^2(\gamma)$ is decomposed into the orthogonal sum
$
L^2(\gamma)=\bigoplus_{k=0}^\infty \mathcal{H}_k.
$

It is clear that $\mathcal{P}^d(\gamma)=\bigoplus_{k=0}^d \mathcal{H}_k$.
The subspaces $\mathcal{H}_k$ can be also defined
by means of multiple Wiener--It\^o stochastic integrals.
This interpretation can be found in \cite[Section 1.1.2]{Nualart}
or in~\cite[Section~2.11]{GM}.

Let us define Sobolev derivatives and gradients of measurable polynomials.
Let $\{e_n\}$ be an orthogonal basis in~$H$. One can assume that $\gamma$
is the countable
power of the standard Gaussian measure on~$\mathbb{R}$ and $\{e_n\}$ is the usual basis
in~$l^2$.
For any $f\in \mathcal {P}^d(\gamma)$ and $p\ge 1$, $r\in\mathbb{N}$,
one can define the Sobolev norm
$$
\| f \|_{p,r}=
\sum_{k=0}^r\biggl(
\int \Bigl( \sum_{i_1,\ldots,i_k} (\partial_{e_{i_1}}\dots\partial_{e_{i_k}} f)^2
\Bigr)^{p/2}\,d\gamma \biggr)^{1/p}
$$
and the Sobolev gradient
$$
\nabla f (x)= \sum_{k=0}^{\infty} \partial_{e_k} f(x) \, e_k,
$$
where $\partial_{e_k}$ is the partial derivative along the vector $e_k$.
One can pick a version of $f$ such that these partial derivatives exist
and $\nabla f(x)\in H$. Moreover, $\| f \|_{p,r}<\infty$ for all $p,r<\infty$.
The Sobolev class $W^{p,r}(\gamma)$ is the completion of $\mathcal {P}^d(\gamma)$
with respect to the norm~$\|\,\cdot\,\|_{p,r}$.
This class coincides also with the completion with respect to the Sobolev norm
of the space of functions of the form $f(l_1(x),\ldots,l_m(x))$, where $f\in C_b^\infty(\mathbb{R}^m)$.
In the case of $X=\mathbb R^n$ and the standard Gaussian
measure $\gamma$ one has $H=X=\mathbb R^n$ and $\nabla f$ is the gradient of $f$ in the usual sense.

As in the finite-dimensional case,
all $\gamma$-measurable polynomials have derivatives of all orders
and the following estimate (the reverse Poincar\'e inequality) holds true:
\begin{equation}\label{rv-Poin}
\int|\nabla f|^2\, d\gamma\le c(d)\int (f - m_f)^2\, d\gamma,
\quad m_f = \int\, f d\gamma.
\end{equation}
This fact follows from the equivalence of all Sobolev norms
and all $L^p$-norms on the space of measurable polynomials
of degree~$d$ (see, e.g., Example 5.3.4 in \cite{GM}).
This equivalence of $L^p$-norms gives a bound
$$
\|f\|_q \le \|f\|_p \le C(p,q,d)\|f\|_q
$$
for all measurable polynomials $f$ of degree $d$ and any $p>q\ge 1$.

For a detailed discussion of $\gamma$-measurable polynomials, see~\cite[Section~5.10]{GM}.

We need the following inequality proved by Carbery and Wright~\cite{CW}
(and also by Nazarov, Sodin, Volberg~\cite{NSV}):
there is an absolute constant $c$ such that,
 for every Gaussian measure {\rm(}more generally, for every  convex measure{\rm)}
 $\gamma$ on $\mathbb{R}^n$ and for every polynomial $f$ of degree~$d$, one has
\begin{equation}\label{cw-ineq}
\gamma(|f|\le t)\biggl(\int|f|\, d\gamma\biggr)^{1/d}\le  cd t^{1/d}, \quad t\ge 0.
\end{equation}

Generalizations to the case of $s$-concave measures are considered in~\cite{BN};
on measurable polynomials on infinite-dimensional locally convex spaces
see also \cite{ArKos}.

We also recall the following known fact about
weakly convergent sequences of distributions of $\gamma$-measurable polynomials
with the same $\gamma$ as above
(more generally, a sequence of polynomials of degree $d$ possessing  uniformly
tight distributions is bounded in all $L^p$, see, e.g., \cite[Exercise~9.8.19]{DM}).

\begin{lemma}\label{moments-weak}
Let $\{f_n\}$ be a sequence of $\gamma$-measurable polynomials of degree~$d$.
Suppose that the distributions $\mu_n=\gamma \circ f_n^{-1}$ converge weakly
to a measure $\mu$ on $\mathbb R$.
Then, for any $p\ge1$, one has convergence of moments
$$
\lim\limits_{n\to\infty} \int_{\mathbb{R}^k} |x|^p\, d\mu_n=\int_{\mathbb{R}^k} |x|^p\, d\mu.
$$
\end{lemma}

\section{Fractional Hardy--Landau--Littlewood type estimates}

Let us give a sufficient condition for membership in the class $B^\alpha(\mathbb{R}^k)$.

\begin{proposition}\label{pro2.1}
Let $\alpha\in (0,1]$.
Let $\nu$ be a Borel measure on $\mathbb{R}^k$.
Suppose that for every function $\varphi\in C_b^\infty(\mathbb{R}^k)$
and every unit vector $e\in\mathbb{R}^k$ one has
$$
\int_{\mathbb{R}^k} \partial_e\varphi(x)\, \nu(dx)\le
C\|\varphi\|_\infty^{\alpha} \|\partial_e\varphi\|_{\infty}^{1-\alpha}.
$$
Then
$$
\|\nu_h-\nu\|_{{\rm TV}} \le 2^{1-\alpha}C|h|^\alpha \quad \forall\, h\in\mathbb{R}^k,
$$
that is, $\nu\in B^\alpha(\mathbb{R}^k)$
and $\|\nu\|_{B^\alpha} \le 2^{1-\alpha}C$.
In particular, the density of $\nu$ belongs to all $L^p(\mathbb{R}^k)$ with $p<k/(k-\alpha)$
according to {\rm(\ref{nikolski-emb})}.
\end{proposition}
\begin{proof}
Let $e = |h|^{-1}h$.
It is easy to see that
\begin{align*}
\|\nu_h-\nu\|_{{\rm TV}}  &=\sup_{\varphi\in C_b^\infty(\mathbb{R}^k),\ \|\varphi\|_\infty \le1}
\int_{\mathbb{R}^k} \varphi(x) \, (\nu_h-\nu)(dx)
\\
&=\sup_{\varphi\in C_b^\infty(\mathbb{R}^k),\ \|\varphi\|_\infty \le1}
\int_{\mathbb{R}^k} [\varphi(x+h)-\varphi(x)]\, \nu(dx)
\\
&=
\sup_{\varphi\in C_b^\infty(\mathbb{R}^k),\ \|\varphi\|_\infty \le1}
\int_{\mathbb{R}^k} \int_0^{|h|}\partial_e\varphi(x+se)\, ds\, \nu(dx).
\end{align*}
Let $\varphi\in C_b^\infty(\mathbb{R}^k)$
and $\|\varphi\|_\infty\le 1$. Consider the function
$$
\Phi(x)=\int_0^{|h|}\varphi(x+se)\, ds.
$$
Note that
$
\sup_{x\in\mathbb{R}^k}|\Phi(x)|\le |h|
$
and
$$
|\partial_e\Phi(x)|=\biggl|\int_0^{|h|}\partial_e\varphi(x+se)ds\biggr|
=|\varphi(x+h)-\varphi(x)|\le2.
$$
By the assumptions of the theorem we have
$$
\int_{\mathbb{R}^k} \partial_e\Phi(x)\,  \nu(dx)\le
C|h|^\alpha 2^{1-\alpha},
$$
hence
$$
\int_{\mathbb{R}^k} \int_0^{|h|}\partial_e\varphi(x+se)\, ds \, \nu(dx)\le
C2^{1-\alpha}|h|^\alpha,
$$
which completes the proof.
\end{proof}

The following result is a fractional analog
of the multidimensional
Hardy--Lan\-dau--Littlewood inequality established in  \cite{BSh}
(in the case $\alpha=1$).

\begin{theorem}\label{t3.2}
Let $\nu, \sigma\in B^\alpha(\mathbb{R}^k)$ be two  Borel probability measures on $\mathbb{R}^k$.
Then
\begin{equation}\label{fracHLL}
\|\sigma-\nu\|_{{\rm TV}} \le
C(k,\alpha)\|\sigma - \nu\|_{B^\alpha}^{1/(1+\alpha)}d_{{\rm K}} (\sigma, \nu)^{\alpha/(1+\alpha)},
\end{equation}
where
$$C(k,\alpha)=1+\int_{\mathbb{R}^k} |x|^\alpha\, \gamma_k(dx).$$
\end{theorem}
\begin{proof}
Let $\gamma_k^{\varepsilon}$ be the centered Gaussian measure on $\mathbb{R}^k$
with the covariance matrix $\varepsilon^2{\rm I}$, i.e.,
with density $(2\pi\varepsilon^2)^{-k/2}\exp(-|x|^/(2\varepsilon^2))$.
By the triangle inequality we have
\begin{equation} \label{ek3.1}
\|\sigma-\nu\|_{{\rm TV}} \le\|(\sigma - \nu)
- (\sigma - \nu)*\gamma_k^\varepsilon\|_{{\rm TV}} +
\|\sigma*\gamma_k^\varepsilon-\nu*\gamma_k^\varepsilon\|_{{\rm TV}} .
\end{equation}
For any function $\varphi\in C_b^\infty(\mathbb{R}^k)$ with $\|\varphi\|_\infty\le1$
the following equalities hold true, where all integrals in this proof
are taken over~$\mathbb{R}^k$:
\begin{align*}
\int\varphi \, d(\sigma*\gamma_k^\varepsilon-\nu*\gamma_k^\varepsilon)&=
\int \varphi(x)\int(2\pi\varepsilon^2)^{-k/2}
\exp\Bigl(-\frac{|y-x|^2}{2\varepsilon^2}\Bigr)\, (\nu-\sigma)(dy)\, dx
\\
&=\int\biggl(\int\varphi(x)(2\pi\varepsilon^2)^{-k/2}
\exp\Bigl(-\frac{|y-x|^2}{2\varepsilon^2}\Bigr)\, dx
\biggr)(\nu-\sigma)(dy).
\end{align*}
Let us consider the function
$$
\Phi(y) := \int\varphi(x)(2\pi\varepsilon^2)^{-k/2}
\exp\Bigl(-\frac{|y-x|^2}{2\varepsilon^2}\Bigr)\, dx.
$$
We have
$$
\nabla\Phi(y) =
\varepsilon^{-1}\int\varphi(y+\varepsilon z)(2\pi)^{-k/2}z
\exp\Bigl(-\frac{|z|^2}{2}\Bigr)\, dz,
$$
hence
$
|\Phi(y)|\le 1, \ |\nabla\Phi(y)|\le\varepsilon^{-1}.
$
Therefore,
\begin{equation} \label{k-kr-differ}
\|\sigma*\gamma_k^\varepsilon-\nu*\gamma_k^\varepsilon\|_{{\rm TV}}
\le \varepsilon^{-1}d_{{\rm K}} (\sigma, \nu).
\end{equation}
We now estimate the remaining term in the right-hand side of (\ref{ek3.1}):
\begin{align*}
&\|(\sigma - \nu)-(\sigma - \nu)*\gamma_k^\varepsilon\|_{{\rm TV}}
\\
&=\sup\limits_{\|\varphi\|_\infty\le1}
\int  \biggl((2\pi\varepsilon^2)^{-k/2}
\exp\Bigl(-\frac{|y|^2}{2\varepsilon^2}\Bigr)
\int\varphi(x)\bigl((\sigma - \nu)-(\sigma_y - \nu_y)\bigr)(dx)\biggr)dy
\\
&\le
\|\sigma - \nu\|_{B^\alpha}
\int(2\pi\varepsilon^2)^{-k/2}
\exp\Bigl(-\frac{|y|^2}{2\varepsilon^2}\Bigr)|y|^\alpha\, dy
\\
&=
\varepsilon^\alpha \|\sigma - \nu\|_{B^\alpha}(2\pi)^{-k/2}
\int |y|^\alpha\exp\Bigl(-\frac{|y|^2}{2}\Bigr)\, dy.
\end{align*}
Hence we have
$$
\|\sigma-\nu\|_{{\rm TV}}
\le
\varepsilon^{-1}d_{{\rm K}} (\sigma, \nu)
+ \varepsilon^\alpha \|\sigma - \nu\|_{B^\alpha} \int |x|^\alpha \gamma_k(dx).
$$
Taking
$\varepsilon = \bigl(\|\sigma - \nu\|_{{\rm K}}/\|\sigma - \nu\|_{B^\alpha}\bigr)^{1/(1+\alpha)}$,
 we obtain~(\ref{fracHLL}).
\end{proof}

\begin{remark}\label{rem2.1}
\rm
(i)
One can modify the previous proof to obtain the following estimate
for probability measures $\nu, \sigma\in B^\alpha(\mathbb{R}^k)$ employing
the Fortet--Mourier metric:
\begin{align*}
\|\sigma-\nu\|_{{\rm TV}}
& \le
C(k,\alpha)\|\sigma - \nu\|_{B^\alpha}^{1/(1+\alpha)}d_{{\rm FM}}
(\sigma, \nu)^{\alpha/(1+\alpha)}+d_{{\rm FM}} (\sigma, \nu)
\\
&\le \bigl(C(k,\alpha)\|\sigma - \nu\|_{B^\alpha}^{1/(1+\alpha)}+2^{1/(1+\alpha)}\bigr)
d_{{\rm FM}} (\sigma, \nu)^{\alpha/(1+\alpha)} ,
\end{align*}
where $C(k,\alpha)$ is the same as above.
To this end,  in place of inequality (\ref{k-kr-differ}) we write
$
\|\sigma*\gamma_k^\varepsilon-\nu*\gamma_k^\varepsilon\|_{{\rm TV}}
\le
\bigl(\varepsilon^{-1}+1\bigr)d_{{\rm FM}} (\sigma, \nu),
$
and then proceed as in the proof above. The additional quantity $2^{1/(1+\alpha)}$ is not needed if
we slightly decrease the power at $d_{{\rm FM}}$ as explained in~(ii).

(ii)
In relation to (i) we observe that the two distances $d_{{\rm FM}}$
and $d_{{\rm K}}$, which in general admit only the one-sided estimate
$d_{{\rm FM}}\le d_{{\rm K}}$, are very close on the set of distributions of polynomials of degree $d$
with variances not exceeding a fixed number~$b$. More precisely, there is a number $L(d,b)$ such that
$$
d_{{\rm K}} (\gamma\circ f^{-1}, \gamma\circ g^{-1})\le
L(d,b) d_{{\rm FM}}(\gamma\circ f^{-1}, \gamma\circ g^{-1})
(|\log d_{{\rm FM}}(\gamma\circ f^{-1}, \gamma\circ g^{-1})|^{d/2}+1).
$$
Indeed, it is known (see \cite[Corollary~5.5.7]{GM}) that
$$
\gamma(x\colon \ |f(x)|\ge t \|f\|_2)\le c_r \exp (- r t^{2/d}), \quad
r<\frac{d}{2e},
$$
where $c_r$ depends only on~$r$.
Let $\varphi$  be a $1$-Lipschitz function on $\mathbb{R}$.
We can assume that $\varphi(0)=0$, since
$\varphi(f)  -\varphi(g)$ does not change if we subtract $\varphi(0)$ from~$\varphi$.
Considering the bounded function $\varphi_R=\max (-R, \min(R,\varphi))$, we obtain
\begin{align*}
&\int_{\mathbb{R}^k} [\varphi(f)  -\varphi(g)]\, d\gamma
\\
&\le (R + 1) d_{{\rm FM}}(\gamma\circ f^{-1}, \gamma\circ g^{-1})
+ \int_{\mathbb{R}^k} \bigl[|\varphi(f)-\varphi_R(f)|+|\varphi(g)-\varphi_R(g)|\bigr]\, d\gamma
\\
 & \le  (R+1) d_{{\rm FM}}(\gamma\circ f^{-1}, \gamma\circ g^{-1})
+ \int_{|f|>R} |f|\, d\gamma +  \int_{|g|>R} |g|\, d\gamma
\\
 & \le (R+1) d_{{\rm FM}}(\gamma\circ f^{-1}, \gamma\circ g^{-1})
+C_1 \exp (- C_2 R^{2/d}).
\end{align*}
Now we take
$$
R=\Bigl(\frac{|\log d_{{\rm FM}}(\gamma\circ f^{-1}, \gamma\circ g^{-1})|}{C_2}\Bigr)^{d/2}
$$
and immediately get the desired estimate if $d_{{\rm FM}}(\gamma\circ f^{-1}, \gamma\circ g^{-1})\le 1$.
Finally, we observe
 that if $d_{{\rm FM}}(\gamma\circ f^{-1}, \gamma\circ g^{-1})> 1$, then
$\exp (- C_2 R^{2/d})<d_{{\rm FM}}(\gamma\circ f^{-1}, \gamma\circ g^{-1})$, and thus
we obtain the estimate in the general case. However, we do not know whether the logarithmic
factor is really needed.
\end{remark}

\begin{remark}\label{rem3}
\rm
Let $\nu\in B^\alpha(\mathbb{R}^k)$ be a  Borel measure on $\mathbb{R}^k$.
Then one can prove by a similar reasoning that for every Borel set $A$ one has
$$
\nu(A) \le C_1(k,\alpha)\|\nu\|_{B^\alpha}^{k/(\alpha+k)}\lambda_k(A)^{\alpha/(\alpha + k)},
$$
where $\lambda_k$ is the standard Lebesgue measure on $\mathbb{R}^k$,
$$
C_1(k,\alpha) = (2\pi)^{-k/2} + (2\pi)^{-k/2}\int_{\mathbb{R}^k}
 \exp{\Bigl(-\frac{|y|^2}{2}\Bigr)}|y|^\alpha dy.
$$
However, the embedding theorem for Nikol'skii--Besov  spaces (see (\ref{nikolski-emb}))
gives a slightly better power: for any $r<\alpha/k$ there is
$C_2(k,\alpha,r)>0$ such that
$$
\nu(A) \le C_2(k,\alpha,r)(\|\nu\|_{B^\alpha}+1)\lambda_k(A)^{r}
\quad
\hbox{for every Borel set $A$.}
$$
\end{remark}

\section{Fractional smoothness of polynomial images of Gaussian measures}

Let us recall that the Ornstein--Uhlenbeck operator $L$ associated with the
standard Gaussian measure $\gamma$ on $\mathbb{R}^n$ is defined by
$$
L\varphi(x)=\Delta \varphi(x)-\langle x,\nabla \varphi(x)\rangle ,
$$
where $\Delta$ is the Laplace operator. The operator $L$ is symmetric
in $L^2(\gamma)$ (with domain $W^{2,2}(\gamma)$) and is frequently used
in the integration by parts formula
$$
\int_{\mathbb{R}^n} \varphi L\psi \, d\gamma
=-\int_{\mathbb{R}^n} \langle \nabla \varphi, \nabla \psi\rangle\, d\gamma .
$$
We employ this formula below.

Let $f\colon\mathbb{R}^n\to\mathbb{R}^k$ be a  mapping such that
its components $f_1,\ldots,f_k$ are polynomials of degree $d$.
Let us introduce the Malliavin matrix of $f$ by
$$
M_f(x)=(m_{i,j}(x))_{i,j\le k},
\quad
m_{i,j}(x):=\langle\nabla f_i(x), \nabla f_j(x)\rangle.
$$
It is a polynomial of degree $2k(d-1)$.
Let
$$
A_f:=(a_{i, j})_{i,j\le k}
$$
 be the adjugate matrix of $M_f$, i.e., $a_{i, j}=M^{j, i}$, where $M^{j, i}$
is the cofactor of $m_{j, i}$ in the matrix~$M_f$. Note that $a^{i,j}$ is a polynomial of degree
$k-1$ in $m_{s,t}$.
Set
$$
\Delta_f:=\det M_f.
$$
We observe that $\Delta_f\ge 0$ and
\begin{equation}\label{inver}
\Delta_f\cdot M_f^{-1} = A_f.
\end{equation}
Let $\sigma^2_{f_i}$ denote the variance of the random variable $f_i$ on $(\mathbb{R}^n,\gamma)$:
$$
\sigma^2_{f_i} := \int_{\mathbb{R}^n}
 \Bigl(f_i - \int_{\mathbb{R}^n} f_i\, d\gamma\Bigr)^2\, d\gamma.
$$

The first main result of this section is the following theorem which says
that the distribution of a  polynomial mapping $f$ with respect to a Gaussian measure
such that $f$ is nondegenerate
(in the sense that $\Delta_f>0$ on a positive measure set, or equivalently, $\gamma_n\circ f^{-1}$
is absolutely continuous)
always belongs to some Nikol'skii--Besov  class whose index depends
only on the maximal degree of components and the number of components, but
not on the number of variables.

\begin{theorem}\label{th4.1}
Let $k,d\in\mathbb{N}$, $a>0$, $b>0$, $\tau>0$. Then there exists a number  $C(d, k, a, b, \tau)>0$
such that, for every mapping $f = (f_1, \ldots, f_k)\colon\,
\mathbb{R}^n\to\mathbb{R}^k$, where each $f_i$ is a polynomial of degree~$d$ and
$$
\int_{\mathbb{R}^n} \Delta_f \, d\gamma_n\ge a, \quad
\max_{i\le k} \sigma_{f_i}\le b,
$$
for every function  $\varphi\in C_b^\infty(\mathbb{R}^k)$
   and every vector $e\in\mathbb{R}^k$ with $|e|=1$,
one has
$$
\int_{\mathbb{R}^n} \partial_e\varphi(f(x))\, \gamma_n(dx)\le C(d, k, a, b, \tau)
\|\varphi\|_\infty^\alpha
\|\partial_e\varphi\|_{\infty}^{1-\alpha}, \quad \alpha=\frac{1}{4k(d-1)+\tau}.
$$
Therefore, we have
$$
\|\gamma_n\circ f^{-1} - (\gamma_n\circ f^{-1})_h\|_{{\rm TV}}
\le C(d, k, a, b, \tau) |h|^\alpha,
$$
equivalently,
$$
\gamma_n\circ f^{-1}\in B^\alpha(\mathbb{R}^k) \quad \text{for every}\ \alpha<\frac{1}{4k(d-1)}.
$$
In particular, the density of $\gamma_n\circ f^{-1}$
belongs to all $L^p(\mathbb{R}^k)$ with $p<k/(k-\alpha)$.
\end{theorem}
\begin{proof}
We can assume that $\|\varphi\|_\infty\le 1$.
If $\|\partial_e\varphi\|_\infty \le1$, then for any $\alpha>0$ we have
(omitting indication of $\mathbb{R}^n$ in all integrations in this proof)
$$
\int \partial_e\varphi(f(x))\, \gamma_n(dx)\le \|\partial_e\varphi\|_\infty \le
\|\partial_e\varphi\|_{\infty}^{1-\alpha}.
$$
Suppose now that $\|\partial_e\varphi\|_{\infty}\ge1$.
It can be easily verified that
$$
M_f (\partial_{x_1}\varphi(f),\ldots, \partial_{x_k}\varphi(f)) =
\bigl(\langle\nabla (\varphi\circ f), \nabla f_1\rangle, \ldots,
\langle\nabla (\varphi\circ f), \nabla f_k\rangle\bigr).
$$
Here the left-hand side is interpreted as the standard product of
a matrix and a vector (with components $\partial_{x_i}\varphi(f)$)
and $\nabla$ denotes the gradient of a function of $n$ variables.
Then by equality (\ref{inver}) we obtain
$$
(\partial_e\varphi)(f) \Delta_f =
\bigl\langle v , A_f e\bigr\rangle, \quad
v=\bigl(\langle\nabla (\varphi\circ f), \nabla f_1\rangle, \ldots,
 \langle\nabla (\varphi\circ f), \nabla f_k\rangle\bigr).
$$
Let $\varepsilon\in (0,1)$ be a fixed number that will be chosen later.
The integral that we want  to estimate can be written as
\begin{equation}\label{ek4.2}
\int \partial_e\varphi(f)\, d\gamma_n
=
\int \partial_e\varphi(f)\frac{\Delta_f}{\Delta_f+\varepsilon}\, d\gamma_n+
\varepsilon\int \frac{\partial_e\varphi(f)}{\Delta_f+\varepsilon}\, d\gamma_n.
\end{equation}
We now estimate each term.
By the reasoning above we can write
$$
\int \partial_e\varphi(f)\frac{\Delta_f}{\Delta_f+\varepsilon}\, d\gamma_n
=
\int \frac{\bigl\langle\bigl(\langle\nabla \varphi\circ f, \nabla f_1\rangle, \ldots,
 \langle\nabla \varphi\circ f, \nabla f_k\rangle\bigr) , A_f e\bigr\rangle}
 {\Delta_f+\varepsilon}\, d\gamma_n.
$$
Letting $h(x) = A_f(x)e$,
we can integrate by parts and write the above term as
\begin{equation} \label{lem3_3terms}
\begin{split}
\int &(\Delta_f+\varepsilon)^{-1}
\sum_{i=1}^k\langle\nabla \varphi\circ f, \nabla f_i\rangle h_i\, d\gamma_n
\\
&=-\sum_{i=1}^k\int\varphi\circ f\Bigl( \frac{h_i Lf_i}
{\Delta_f+\varepsilon} -
\frac{h_i\langle\nabla f_i, \nabla \Delta_f\rangle}{(\Delta_f+\varepsilon)^{2}}
+\frac{\langle\nabla f_i, \nabla h_i \rangle}{\Delta_f+\varepsilon}\Bigr)\,
d\gamma_n
\\
&\le
\int \Bigl|\sum_{i=1}^kh_i Lf_i\Bigr|
(\Delta_f+\varepsilon)^{-1}\, d\gamma_n
+\int\Bigl| \sum_{i=1}^kh_i\langle\nabla f_i, \nabla \Delta_f\rangle
\Bigr|(\Delta_f+\varepsilon)^{-2}\, d\gamma_n
 \\
&+
\int\Bigl| \sum_{i=1}^k\langle\nabla f_i, \nabla h_i \rangle
\Bigr|(\Delta_f+\varepsilon)^{-1}\, d\gamma_n.
\end{split}
\end{equation}
We have to estimate each of the three terms. First of all,
note that $\Delta_f$ is itself a measurable polynomial of degree $2k(d-1)$.
We set
$$
\beta=\frac{1}{2k(d-1)}
$$
and use the Carbery--Wright inequality (\ref{cw-ineq}) to obtain
\begin{equation}\label{CaW4.4}
\begin{split}
&\int (\Delta_f+\varepsilon)^{-p}\, d\gamma_n=
p\int_0^{\varepsilon^{-1}}t^{p-1}\gamma_n\bigl((\Delta_f+\varepsilon)^{-1}\ge t\bigr)\, dt
\\
&=
p\int_0^\infty (s+\varepsilon)^{-p-1}\gamma_n\bigl(\Delta_f \le s\bigr)\, ds
\\
&\le2cpk(d-1)\biggl(\int \Delta_f \, d\gamma_n\biggr)^{-\beta}
\int_0^\infty (s+\varepsilon)^{-p-1}s^{\beta}\, ds
\\
&=\varepsilon^{-p+\beta}2cpk(d-1)
\biggl(\int \Delta_f\, d\gamma_n\biggr)^{-\beta}\int_0^\infty(s+1)^{-p-1}s^{\beta}\, ds.
\end{split}
\end{equation}
Let
$$
c(p,d) := \biggl(2cpk(d-1)\int_0^\infty(s+1)^{-p-1}s^{\beta} \, ds\biggr)^{1/p}.
$$
Let $\|A\|_{HS}=\Bigl(\sum_{i,j} a_{ij}^2\Bigr)^{1/2}$
denote the Hilbert--Schmidt norm of a matrix~$A=(a_{ij})$.
Then $\|A_f(x)\|_{HS}$
is estimated by a polynomial in the matrix elements $m_{i,j}(x)$. Hence its
$L^p$-norms are bounded by powers of $b$ (with some constants depending on $d$, $k$
and~$p$).
Let us estimate the first term in the right-hand side of (\ref{lem3_3terms}):
\begin{align*}
\int \Bigl| \sum_{i=1}^kh_i Lf_i\Bigr|
 (\Delta_f+\varepsilon)^{-1}\, d\gamma_n
 &\le
 \int (\Delta_f+\varepsilon)^{-1}\|A_f\|_{HS}
 \Bigl(\sum_{i=1}^k |Lf_i|^2\Bigr)^{1/2}\,
 d\gamma_n
 \\
 &\le \varepsilon^{-1}\int\|A_f\|_{HS}
 \Bigl(\sum_{i=1}^k |L f_i|^2\Bigr)^{1/2}\, d\gamma_n.
\end{align*}
Next we estimate the second term in the right-hand side of (\ref{lem3_3terms}):
\begin{multline*}
\int\Bigl| \sum_{i=1}^kh_i\langle\nabla f_i, \nabla \Delta_f\rangle\Bigr|
(\Delta_f+\varepsilon)^{-2}\, d\gamma_n
\\
\le\int (\Delta_f+\varepsilon)^{-2}
\|A_f\|_{HS} \Bigl( \sum_{i=1}^k\langle\nabla f_i, \nabla \Delta_f\rangle^2\Bigr)^{1/2}\, d\gamma_n
\\
\le
\int (\Delta_f+\varepsilon)^{-2}\|A_f\|_{HS}|\nabla \Delta_f|
\Bigl( \sum_{i=1}^k|\nabla f_i|^2\Bigr)^{1/2}\, d\gamma_n
\\
\le
\biggl(\int(\Delta_f+\varepsilon)^{-2q}\, d\gamma_n\biggr)^{1/q}
\biggl(\int\|A_f\|^{q'}_{HS}|\nabla \Delta_f|^{q'}
\Bigl( \sum_{i=1}^k|\nabla f_i|^2\Bigr)^{q'/2}\, d\gamma_n\biggr)^{1/q'}
\\
\le
c(2q,d)^2 \varepsilon^{-2+\beta/q}
\biggl(\int \Delta_f \, d\gamma_n\biggr)^{-\beta/q}
\\
\times
\biggl(\int\|A_f\|^{q'}_{HS}|\nabla \Delta_f|^{q'}
\Bigl( \sum_{i=1}^k|\nabla f_i|^2\Bigr)^{q'/2}\, d\gamma_n\biggr)^{1/q'},
\end{multline*}
where $q'=q/ (q-1)$ appears due to H\"older's inequality.

Finally, let us estimate the third term in the right-hand side of (\ref{lem3_3terms}):
\begin{align*}
\int\Bigl| \sum_{i=1}^k\langle\nabla f_i, \nabla h_i \rangle\Bigr|(\Delta_f
+\varepsilon)^{-1}\, d\gamma_n
& \le
\int (\Delta_f+\varepsilon)^{-1}\sum_{i=1}^k|\nabla f_i|\, |\nabla h_i|\, d\gamma_n
\\
&\le
\frac{1}{2\varepsilon}\int\sum_{i=1}^k \Bigl(|\nabla f_i|^2+|\nabla h_i|^2\Bigr)\, d\gamma_n.
\end{align*}
Since
$-2+\beta/q<-1$ and $\varepsilon\le1$, we have
$\varepsilon^{-1}\le \varepsilon^{-2+\beta/q}$.

We now use (\ref{CaW4.4})
to estimate the second term in the right-hand side of (\ref{ek4.2}):
$$
\int \frac{\partial_e\varphi(f)}{\Delta_f+\varepsilon}\, d\gamma_n
\le
\|\partial_e\varphi\|_\infty c(1,d)\varepsilon^{-1+\beta}\biggl(\int \Delta_f \,d\gamma_n\biggr)^{-\beta}.
$$
Setting  $\tau=\frac{q-1}{q}$ and taking
$$
\varepsilon = \|\partial_e\varphi\|_\infty^\omega, \quad
\omega={-\frac{1}{2+\tau\beta}}={-\frac{2k(d-1)}{4k(d-1)+\tau}},
$$
we arrive at the estimate
\begin{equation}\label{n-est-j}
\int \partial_e\varphi(f)\, d\gamma_n\le C\|\partial_e\varphi\|_\infty^{1-\alpha},
\quad  \alpha=\frac{1}{4k(d-1)+\tau},
\end{equation}
where
\begin{align*}
C = &\int\|A_f\|_{HS}\Bigl(\sum_{i=1}^k |Lf_i|^2\Bigr)^{1/2}
\, d\gamma_n
\\
&+c(2q,d)^2\biggl(\int \Delta_f\, d\gamma_n\biggr)^{-\beta/q}
\biggl(\int\|A_f\|^{q'}_{HS}|\nabla \Delta_f|^{q'}
\Bigl( \sum_{i=1}^k|\nabla f_i|^2\Bigr)^{q'/2}\, d\gamma_n\biggr)^{1/q'}
\\
&+\frac{1}{2}\int\sum_{i=1}^k \Bigl[|\nabla f_i|^2+|\nabla h_i|^2\Bigr]\, d\gamma_n+
c(1,d)\biggl(\int \Delta_f \, d\gamma_n\biggr)^{-\beta}.
\end{align*}
Using inequality (\ref{rv-Poin}) and the equivalence of the $L^p$-norms of measurable polynomials
of degree~$d$
we can replace this number
$C$ by a number $C(d, k, a, b, \tau)$ that  depends only on $d$, $k$, $a$, $b$ and~$\tau$.
Recall that $\|A_f\|_{HS}$ is estimated by a polynomial in the matrix elements $m_{i,j}(x)$. Hence its
$L^p$-norms are also bounded by powers of~$b$.
By choosing $q>1$ sufficiently close to $1$, we can make
$\tau=\frac{q-1}{q}$ in (\ref{n-est-j}) as small as we   wish. It remains to take into
account Proposition \ref{pro2.1}.
\end{proof}

By the aforementioned compact embedding (\ref{nikolski-emb}) on balls,
we immediately obtain convergence of densities in $L^p(\mathbb{R}^k)$
with $p<k/(k-\alpha)$
in case of weak convergence of distributions of mappings
satisfying the assumptions of Theorem \ref{th4.1}
(which sharpens a result from~\cite{NNP}).

Combining Theorem \ref{th4.1} and Theorem \ref{t3.2} we obtain the following result.

\begin{theorem}\label{t4.3}
Let $k,d\in \mathbb{N}$, $a>0$, $b>0$,  $\tau>0$. Then there is
$C=C(d, k, a, b, \tau)$ such that, whenever
$
f = (f_1, \ldots, f_k)
$
and
$
g = (g_1, \ldots, g_k)
$
are mappings from $\mathbb{R}^n$ to $\mathbb{R}^k$ such that
their components $f_i, g_i$ are polynomials of  degree~$d$ with
$$
\int \Delta_f\, d\gamma_n\ge a,\quad
\int \Delta_g\, d\gamma_n\ge a, \quad \max_{i\le k} \sigma_{f_i}\le b,
\quad
 \max_{i\le k} \sigma_{g_i}\le b,
$$
one has
$$
d_{{\rm TV}} (\gamma_n\circ f^{-1}, \gamma_n\circ g^{-1})\le
C d_{{\rm K}} (\gamma_n\circ f^{-1}, \gamma_n\circ g^{-1})^{\theta},
\quad \theta=\frac{1}{4k(d-1)+1+\tau}.
$$
\end{theorem}

\begin{remark}\label{rem3.1}
\rm
Using Remark \ref{rem2.1}, one can replace $d_{K}$ with $d_{{\rm FM}} $, that is,
under the assumptions of the theorem the following estimate is also true:
$$
d_{{\rm TV}} (\gamma_n\circ f^{-1}, \gamma_n\circ g^{-1})
\le C d_{{\rm FM}} (\gamma_n\circ f^{-1}, \gamma_n\circ g^{-1})^{\theta},
\quad \theta=\frac{1}{4k(d-1)+1+\tau}
$$
for every $\tau>0$ and some other number $C=C(d, k, a, b, \tau)$.
\end{remark}

We observe that the constants in Theorems \ref{th4.1} and \ref{t4.3} do not depend
on the dimension $n$. Hence
these theorems hold true when $f_i\colon\, X\to\mathbb R$ are $\gamma$-measurable
polynomials with respect to  an arbitrary centered
Radon Gaussian measure $\gamma$ on a locally convex space $X$.

\begin{corollary}\label{newcor1}
Let $\gamma$ be a centered
Radon Gaussian measure  on a locally convex space~$X$.
Let $k,d\in\mathbb{N}$, $a>0$, $b>0$, $\tau>0$. Then there is $C(d, k, a, b, \tau)>0$
such that, for every mapping $f = (f_1, \ldots, f_k)\colon\,
X\to\mathbb{R}^k$, where each $f_i$ is a $\gamma$-measurable
polynomial of degree~$d$ and
$$
\int_{\mathbb{R}^n} \Delta_f \, d\gamma> a, \quad \max_{i\le k}\sigma_{f_i}\le b,
$$
for every function  $\varphi\in C_b^\infty(\mathbb{R}^k)$
 and every vector $e\in\mathbb{R}^k$ with $|e|=1$,
one has
$$
\int_{X} \partial_e\varphi(f(x))\, \gamma(dx)\le C(d, k, a, b, \tau)
\|\varphi\|_\infty^{\alpha}
\|\partial_e\varphi\|_{\infty}^{1-\alpha}, \quad \alpha=\frac{1}{4k(d-1)+\tau}.
$$
Therefore, if $\Delta_f>0$ on a positive measure set, the induced
measure $\gamma\circ f^{-1}$ belongs to the Nikol'skii--Besov  class
$B^\alpha(\mathbb{R}^k)$ with $\alpha$ that depends only on $d$ and~$k$.
\end{corollary}
\begin{proof}
By the Tsirelson isomorphism theorem (see \cite[Chapter~3]{GM}),
we can assume that $\gamma$ is the countable power of the standard Gaussian
measure on the real line (i.e., $\gamma$ is defined
on~$\mathbb{R}^\infty$). In that case we can approximate each polynomial
$f_i$  by the sequence of its finite-dimensional
conditional expectations $f_{i,n}$ with respect
to the $\sigma$-fields generated by the first $n$ variables $x_1,\ldots,x_n$.
Recall that
$$
f_{i,n}(x_1,\ldots,x_n)=\int_X f_i(x_1,\ldots,x_n, y)\, \gamma(dy),
$$
where we write vectors in $\mathbb{R}^\infty$ in the form
$(x_1,\ldots,x_n,y)$, $y=(y_1,y_2,\ldots)\in \mathbb{R}^\infty$.
It is well-known that each $f_{i,n}$ is a polynomial of degree~$d$
(see \cite[Proposition 5.4.5 and Proposition~5.10.6]{GM}).
Moreover, the polynomials $f_{i,n}$ converge to $f_i$ almost everywhere
and in all Sobolev norms
(see \cite[Corollary 3.5.2 and Proposition 5.4.5]{GM}). Therefore, for the corresponding
mappings $f_n=(f_{1,n},\ldots,f_{k,n})$ the integrals of $\Delta_{f_n}$
are not less than $a$
for all $n$ sufficiently large.
In addition, $\sigma_{f_{i,n}}\le \sigma_{f_{i}}\le b$.
This enables us to pass to the limit
$n\to\infty$ in the inequality in Theorem~\ref{th4.1}.
\end{proof}

Similarly we obtain the following result.

\begin{corollary}\label{c4.1}
Let $\gamma$ be a centered Radon Gaussian measure on a locally convex space~$X$.
Let $k,d\in \mathbb{N}$, $a>0$, $b>0$,  $\tau>0$ be fixed.
Then there exists a number
$C_1=C_1(d, k, a, b, \tau)$ such that, whenever
$$
f = (f_1, \ldots, f_k)\quad \hbox{and}\quad g = (g_1, \ldots, g_k)
$$
are mappings from $X$ to $\mathbb{R}^k$ such that
their components $f_i, g_i$ are $\gamma$-measurable polynomials of degree~$d$ with
$$
\int \Delta_f\, d\gamma\ge a,\quad
\int \Delta_g\, d\gamma\ge a, \quad \sigma_{f_i}\le b,\quad \sigma_{g_i}\le b,
\quad i=1,\ldots, k,
$$
one has
$$
d_{{\rm TV}} (\gamma\circ f^{-1}, \gamma\circ g^{-1})\le
C_1 d_{{\rm FM}} (\gamma\circ f^{-1}, \gamma\circ g^{-1})^{1/(4k(d-1)+1+\tau)}.
$$
\end{corollary}

Along with Lemma \ref{moments-weak} this yields the following fact.

\begin{corollary}\label{c4.2}
Let $\gamma$ be a Radon Gaussian measure on a locally convex space~$X$.
Let $f_n=(f_{1,n},\ldots,f_{k,n})\colon\, X\to \mathbb{R}^k$ be a sequence of mappings
such that each $f_{j,n}$ is a $\gamma$-measurable polynomial of degree~$d$.
Suppose that the distributions $\gamma \circ f_n^{-1}$ converge weakly
on~$\mathbb R^k$ and there is $a>0$ such that for all $n\in\mathbb N$
$$
\int \Delta_{f_n}\,d\gamma >a.
$$
Then these measures also converge in variation
and, for every $\tau>0$, there exists a number
$C_2$, depending on $d, k, a, \tau$, and a common bound for the variances
of the components of~$f_n$,
such that
$$
d_{{\rm TV}} (\gamma\circ f_m^{-1}, \gamma\circ f_n^{-1})\le
C_2 d_{{\rm FM}} (\gamma\circ f_m^{-1}, \gamma\circ f_n^{-1})^{1/(4k(d-1)+1+\tau)}.
$$
\end{corollary}

This is a multidimensional generalization of \cite[Theorem 3.1]{NP}
and an improvement of the rate of convergence as compared
to \cite[Theorem~4.1]{NNP}.

It is worth noting that, as was shown in \cite{NNP} extending
a~result from \cite{Kus},
 a~polynomial mapping $f$ from an infinite-dimensional space
with a Gaussian measure $\gamma$ to $\mathbb{R}^k$ has an absolutely continuous
distribution precisely when $\Delta_f$ is not zero a.e.
(equivalently, $\Delta_f>0$ on a positive measure set due to the $0-1$ law for
polynomials, see \cite[Proposition~5.10.10]{GM}). Moreover,
$\gamma\circ f^{-1}$ is not absolutely continuous precisely when there
is a polynomial $Q$ on $\mathbb{R}^k$ such that $Q(f)$ is a constant a.e.
Therefore,
the assumed
lower bound on the expectations of $\Delta_f$ and $\Delta_g$ is quite natural.

Combining Theorem \ref{th4.1} and Remark \ref{rem3}, one can obtain the following theorem,
which in a sense generalizes the Carbery--Wright inequality
(but the latter has been used in the proof).

\begin{corollary}
Let $k,d\in \mathbb{N}$, $a>0$, $b>0$,  $\tau>0$. Then there is
$C=C(d, k, a, b, \tau)$ such that if
$f = (f_1, \ldots, f_k)\colon\, \mathbb{R}^n\to\mathbb{R}^k$,
where each $f_i$ is a polynomial of degree~$d$, satisfies the conditions
$$
\int \Delta_f \, d\gamma_n\ge a,\quad
\max_{i\le k} \sigma_{f_i}\le b ,
$$
then
$$
\gamma_n(f\in A) \le C(d, k, a, b, \tau)\lambda_k(A)^{\theta},
\quad \theta=\frac{1}{4k^2(d-1)+\tau},
$$
where $\lambda_k$
is the standard Lebesgue measure on~$\mathbb{R}^k$.
\end{corollary}

Let us mention a result from \cite{BZ15}
on distributions of multidimensional random vectors the components
of which are general functions belonging to the Sobolev classes $W^{p,2}(\gamma)$, where
$\gamma$ is a general centered Radon Gaussian measure.
Suppose we are given a sequence of mappings
$$
F_n=(F^1_n,\ldots,F^k_n)\colon\, X\to \mathbb{R}^k
$$
such that  $F^i_n\in W^{4k,2}(\gamma)$.
Let $\mu_n=\gamma\circ F_n^{-1}$.
The following theorem proved in \cite{BZ15}
is based on a simple observation
that by the compactness of the embedding of the space
$BV(U)$ of functions of bounded variation on a ball $U\subset \mathbb{R}^k$ to the space $L^1(U)$,
every weakly convergent sequence of nonnegative
measures $\mu_n$ on  $U$  with densities bounded in the norm of $BV(U)$
converges also in variation.
In order to obtain from this convergence in variation on the whole space,
it is necessary to add the uniform tightness of the measures~$\mu_n$,
i.e., the condition $\lim\limits_{R\to\infty} \sup_n \mu_n(\mathbb{R}^k\backslash U_R)=0$,
where $U_R$ is the closed ball of radius $R$  centered at the origin.
In our  situation the uniform tightness
follows from the estimate $\sup_{n,i} \|F_n^i\|_{L^1(\gamma)}<\infty$,
which gives the estimate
$$
\sup_n \int_{\mathbb{R}^k} |x|\, \mu_n(dx)<\infty.
$$
The assumption of the theorem is chosen in such a way that we are able to apply the indicated
reasoning not to the original sequence of induced measures $\mu_n$, but to some sequence
asymptotically approaching it.  For the reader's convenience and also taking
into account that the condition in \cite{BZ15} contains a misprint
(the considered norm in Theorem~2 and Corollary~1 in \cite{BZ15}
should be $\|F^i_n\|_{4d,2}$, not $\|F^i_n\|_{2d,2}$),
we include the proof that is not long.
Set
 $$
 \delta(\varepsilon)= \sup\limits_n \ \gamma(\Delta_{F_n}\le \varepsilon).
 $$

\begin{theorem}
Suppose that
$$
\sup_n \|F^i_n\|_{4k,2}<\infty \quad \hbox{and}\quad \lim\limits_{\varepsilon \to 0}\delta (\varepsilon) = 0.
$$
 Then the sequence of measures $\mu_n=\gamma\circ F_n^{-1}$ has a subsequence convergent in variation.
\end{theorem}
\begin{proof}
Let us consider the measures
$$
\nu_{n,\varepsilon}=\Bigl(\frac{\Delta_n}{\Delta_n+\varepsilon^2}\cdot\gamma\Bigr)\circ F_n^{-1},
\quad \Delta_n:=\Delta_{F_n}, \ \varepsilon>0.
$$
Let $\varphi\in C_0^\infty(\mathbb R^k)$.
Applying (\ref{inver}) and using the notation
$m^n_{i,j}$ and $a^n_{i,j}$ for the elements of $M_{F_n}$ and $A_{F_n}$,
respectively, we obtain
\begin{align}\label{dan-e}
\int_X \partial_{x_i} \varphi \, d \nu_{n,\varepsilon}&=
\int_X (\partial_{x_i} \varphi(F_n)) \frac{\Delta_n}{\Delta_n+\varepsilon^2}
\, d\gamma
=\int_X \sum_{j,l} \frac{ a_{i,j}^n}{\Delta_n+\varepsilon^2}
\, m_{j,l}^n (\partial_{x_l} \varphi)(F_n)  \, d\gamma  \notag
\\
&=\int_X \sum_{j} \frac{ a_{i,j}^n}{\Delta_n+\varepsilon^2} \,
\langle \nabla (\varphi \circ F_n), \nabla F^j_n\rangle_H  \, d\gamma.
\end{align}
It is known (see \cite[Section~5.8]{GM} or \cite[Section~4.2]{Shig})
that for every function $v$ in the second Sobolev class $W^{p,2}(\gamma)$, where $p>1$,
and every function $g\in W^{p',1}(\gamma)$, where $p'=p/(p-1)$,
 one has the following integration by parts  formula:
$$
\int_X \langle\nabla g, \nabla v\rangle_H\, d\gamma
=-\int_X g Lv\, d\gamma,
$$
where $Lv\in L^p(\gamma)$ is the extension of the Ornstein--Uhlenbeck operator to $W^{p,2}(\gamma)$.
Hence for all $g\in W^{qp',1}(\gamma)$ and $\psi\in W^{q'p',1}(\gamma)$
with  $q>1$, $q'=q/(q-1)$ we have (since $\psi g \in W^{p',1}(\gamma)$)
$$
 \int_X \langle\nabla g, \nabla v\rangle_H \psi\, d\gamma
 =-\int_X [g\psi Lv + g\langle\nabla \psi,\nabla v\rangle_H]\, d\gamma .
 $$
We are going to apply this formula to (\ref{dan-e}).
The hypothesis of the theorem implies that
$$
v=F^j_n \in W^{4k,2}(\gamma), \quad
g=\varphi \circ F_n \in W^{4k,1}(\gamma).
$$
To apply the integration by parts formula, we only need to ensure that
$$
\psi=\frac{ a_{i,j}^n}{\Delta_n+\varepsilon^2}\in W^{s,1}(\gamma),\quad s=\frac{4k}{4k-2}.
$$
The $L^s(\gamma)$-norm of $\psi$ is finite, since
$|\psi| \le |a^n_{i,j}|/\varepsilon^2$ and
$$
\|a^n_{i,j}\|_{L^s}  \le C \sum_{l,r\ne i,j} \| \langle\nabla F_n^l,\nabla F_n^r\rangle_H\|_{L^{s(k-1)}}
\,d\gamma \le
C \sum_{l}\|F^l_n\|_{s(2k-2),2}.
$$
The right-hand side is finite, because
$$
s\cdot(2k-2)=(4k(k-2))/(2k-1)<4k
\quad\hbox{and}\quad
\sup_n \|F^i_n\|_{4k,2}<\infty
$$
 by the assumption of the theorem.
Thus, $\psi \in L^s(\gamma)$.

Next, we show that $|\nabla \psi|_H \in L^s(\gamma)$.
Using the cofactor expansion for the determinant $\Delta_n$, we see that
$$
\nabla \Delta_n=\nabla \det {M_{F_n}}=
\sum_{i,j} \frac{\partial \det M_{F_n}}{\partial m^n_{i,j}} \nabla m^n_{i,j}=
\sum_{i,j}  a^n_{i,j} \nabla m^n_{i,j}
$$
and thus
$$
\nabla \psi=\sum_{i,j}\Bigl[ \frac{\nabla a^n_{i,j}}{\Delta_n+\varepsilon^2}+
\frac{a^n_{i,j}}{(\Delta_n+\varepsilon^2)^2} \sum_{k,r} a^n_{k,r}\nabla m^n_{k,r}\Bigr] .
$$
Similarly to the calculations above we prove that
$$
\bigl\| |\nabla a^n_{i,j}|_H \bigr\|_{L^s} \le
C \sum_{l}\|F^l_n\|_{s(2k-2),2}
$$
and
$$
\bigl\| |a^n_{i,j}a^n_{k,r}\nabla m^n_{k,r}|_H \bigr\|_{L^s}  \le
C \sum_{t}\|F^t_n\|_{s(4k-2),2}=C \sum_{t}\|F^t_n\|_{4k,2}.
$$
Thus, $|\nabla \psi|_H \in L^s(\gamma)$ and $\psi \in W^{s,1}(\gamma)$, as announced.

Applying the integration by parts formula to (\ref{dan-e}), we obtain
\begin{multline*}
\int_{\mathbb{R}^d} \partial_{x_i} \varphi \, d \nu_{n,\varepsilon}=
\sum_{j}
\int_X  \frac{ a_{i,j}^n}{\Delta_n+\varepsilon^2} \,
\langle\nabla (\varphi \circ F_n), \nabla F_j\rangle_H  \, d\gamma
\\
=
-\sum_{j} \int_X  \varphi (F_n)\, \frac{a_{i,j}^n}{\Delta_n+\varepsilon^2}
 \, L F_j  \, d\gamma
-
\sum_{j} \int_X  \varphi (F_n) \,\Bigl \langle\nabla  \frac{ a_{i,j}^n}{\Delta_n
+\varepsilon^2}, \nabla F_j \Bigr\rangle_H  \, d\gamma.
\end{multline*}
Hence the generalized partial derivatives of the measure $\nu_{n,\varepsilon}$
are the measures
$$
\sum_j \Bigl(
\frac{ a_{i,j}^n}{\Delta_n+\varepsilon^2} \, L F_j
 + \Bigl\langle\nabla \frac{ a_{i,j}^n}{\Delta_n+\varepsilon^2}, \nabla F_j \Bigr\rangle_H
\Bigr)\, \gamma \circ F_n^{-1}.
$$
Therefore,  the measure $\nu_{n,\varepsilon}$ has a density
$\varrho_{n,\varepsilon }$  of class $BV$ and its  $BV$-norm is dominated by
$$
1+
\Bigl \|\sum_j \Bigl(
\frac{ a_{i,j}^n}{\Delta_n+\varepsilon^2} \, L F_j
+ \Bigl\langle\nabla  \frac{ a_{i,j}^n}{\Delta_n+\varepsilon^2}, \nabla F_j \Bigr\rangle_H
\Bigr) \Bigr\|_{L^1(\gamma)}\le M(\varepsilon).
$$
It is known  that the
embedding $BV(U_R) \to L_1(U_R)$  is compact, where $U_R$ is the
ball of radius  $R$ centered at the origin in $\mathbb R^d$.
Hence there exists a subsequence $\{i_n\}$ such that $\{\varrho_{i_n, 1/m}\}$
converges in $L_1(U_m)$ for every $m\in \mathbb N$.

Let us estimate  $\| \nu_{i,\varepsilon}- \mu_i\|$ in the following way:
\begin{equation}\label{dif1}
\| \nu_{i,\varepsilon}- \mu_i\| = \int
\frac {\varepsilon^2} {\Delta_i+\varepsilon^2} \, d\gamma \le
 \varepsilon+ \gamma(\Delta_i \le \varepsilon) \le
\varepsilon+ \delta (\varepsilon).
\end{equation}
We observe that the family of measures $\{\nu_{i,\varepsilon}\}$, where
 $i\ge 1$, $\varepsilon>0$, is  uniformly tight.
 This follows by the  boundedness of $\{F_n$\} in $L^1(\gamma)$
and the Chebyshev inequality.

Let us now show  that the sequence of measures $\mu_{i_n}$ is fundamental in variation.
Let $\varepsilon>0$. Using the uniform tightness and (\ref{dif1})
we take $M$ such that
$$
\| \nu_{i,1/M}- \mu_i\| \le \varepsilon/5, \quad
\nu_{i,\delta}(\mathbb R^d \backslash U_{M})
\le \varepsilon / 5 \quad \forall \delta>0.
$$
Next, we take $N$ such that for all $n,m>N$ we obtain
$$
\|\varrho_{i_n, 1/M}-\varrho_{i_m, 1/M}\|_{L_1(U_{M})} \le \varepsilon / 5.
$$
 Then for all $n,m>N$ we have
\begin{align*}
\|\mu_{i_n}-\mu_{i_m} \| &\le \|\nu_{i_n,1/M}-\nu_{i_m,1/M} \|
+ \frac{2\varepsilon}{5}
=\|\varrho_{i_n,1/M}-\varrho_{i_m,1/M} \|_{L_1({\mathbb R^d})}
+ \frac{2\varepsilon}{5}
 \\
& \le \|\varrho_{i_n, 1/M}-\varrho_{i_m, 1/M}\|_{L_1(U_{M})}
+  \frac{4\varepsilon}{5} \le \varepsilon,
\end{align*}
The theorem is proved.
\end{proof}

\begin{corollary}
If a sequence $\{F^i_n\}$ is bounded in $W^{4k,2}(\gamma)$ and
$\delta (\varepsilon) \to 0$ and the distributions of $F_n$ converge weakly, then
they converge in variation.
\end{corollary}

This corollary provides another proof of
the already known fact
that if we have  $F^i_n \in \mathcal{P}_d $ and  $\|\Delta_n\|_1 \ge \beta>0$ and
the sequence of  distributions of $F_n$ converges weakly,  then it converges in variation.

\section{The one-dimensional case}

In the one-dimensional case (i.e., $k=1$) one can obtain some better estimates.
They are derived from the following theorem that replaces Theorem \ref{th4.1}
in this case
and a similar result in Theorem~\ref{t-new} that yields an even better fractional order
at the Kantorovich norm (namely, $1/(d+1)$),
but at the cost of a worse constant. As above, $\gamma_n$ is the standard Gaussian
measure on~$\mathbb{R}^n$.

\begin{theorem}\label{t5.1}
Let $d\in \mathbb{N}$, $\tau>0$. Then there is a number
$C(d,\tau)>0$ such that, whenever
$f\colon\,\mathbb{R}^n\to\mathbb{R}$ is a polynomial of degree~$d$,
for all $\varphi\in C_b^\infty(\mathbb{R}^1)$
one has
$$
\int_{\mathbb{R}^n}
 \varphi'(f(x))\, \gamma_n(dx)\le C(d,\tau)
 \sigma_f^{-\alpha} \|\varphi\|_\infty^{\alpha} \|\varphi'\|_{\infty}^{1-\alpha},
\quad \alpha=\frac{1}{2d-2+\tau}.
$$
Therefore, $\gamma_n\circ f^{-1}$ belongs to the Nikol'skii--Besov  class~$B^\alpha(\mathbb{R})$
independent of~$n$, provided that $f$ is not a constant.
\end{theorem}
\begin{proof}
We can assume that $\|\varphi\|_\infty\le 1$.
Fix $\varepsilon>0$ (which will to be chosen later).
The integral that we want  to estimate equals
(we again omit indication of $\mathbb{R}^n$ in the integrals below)
\begin{equation} \label{ek5.1}
\begin{split}
\int \varphi'(f(x))\, \gamma_n(dx)=&
\int \varphi'(f(x))\frac{\langle\nabla f(x),\nabla f(x)\rangle }
{\langle\nabla f(x),\nabla f(x)\rangle+\varepsilon}\, \gamma_n(dx)
\\
&+\varepsilon\int \frac{\varphi'(f(x))}{\langle\nabla f(x),\nabla f(x)\rangle+\varepsilon}
\, \gamma_n(dx).
\end{split}
\end{equation}
Let us estimate every term.
For the first term, integrating by parts, we have
\begin{equation}\label{ek5.2}
\begin{split}
\int \varphi'(f) & \frac{\langle\nabla f,\nabla f\rangle}{\langle\nabla f,\nabla f\rangle
+\varepsilon}\, d\gamma_n=
\int \frac{\langle\nabla \varphi\circ f, \nabla f\rangle} {\langle\nabla f,\nabla f\rangle
+\varepsilon}\, d\gamma_n
\\
=-&\int\varphi(f) \Bigl(\frac{ Lf}{\langle\nabla f,\nabla f\rangle
+\varepsilon}
-\frac{\langle D^2 f\cdot \nabla f,\nabla f\rangle}{(\langle\nabla f,\nabla f\rangle+\varepsilon)^2}
\Bigr)\, d\gamma_n
\\
&\le
\int \frac{| Lf|}{\langle\nabla f,
\nabla f\rangle+\varepsilon} \, \gamma_n +
\int \frac{\|D^2 f\|_{HS}}{\langle\nabla f,\nabla f\rangle+\varepsilon}\, d\gamma_n
\\
&\le
\bigl(\|Lf\|_{L^{q'}(\gamma_n)}
+\|D^2 f\|_{L^{q'}(\gamma_n)}\bigr)
\biggl(\int \bigl(\langle\nabla f,\nabla f\rangle
+\varepsilon\bigr)^{-q}\, d\gamma_n\biggr)^{1/q},
\end{split}
\end{equation}
where $q>1$. Set
$$
\beta=\frac{1}{2(d-1)}.
$$
Using inequality (\ref{rv-Poin}) and the equivalence of the Sobolev and
$L^p$-norms of polynomials of degree~$d$, we obtain that
$$
\| Lf\|_{L^{q'}(\gamma_n)}
+\|D^2 f\|_{L^{q'}(\gamma_n)}\le C(d,q)\sigma_f.
$$
Using (\ref{CaW4.4}), we obtain that the last expression in (\ref{ek5.2})
is not greater than
$$
C(d,q)\sigma_f\varepsilon^{-1+\beta/q}
\biggl(\int \langle\nabla f,\nabla f\rangle\, d\gamma_n\biggr)^{-\beta/q},
$$
which by the Poincar\'e inequality is not greater than
$$
c_1(d,q)\sigma_f^{1-2\beta/q}\varepsilon^{-1+\beta/q}.
$$
Now let us estimate the second term in the right-hand side of (\ref{ek5.1}).
As above, using (\ref{CaW4.4}) and the Poincar\'e inequality, we obtain
$$
\int (\langle\nabla f,\nabla f\rangle+\varepsilon)^{-1}\, d\gamma_n\le
c(d)\sigma_f^{-2\beta}\varepsilon^{-1+\beta}.
$$
Therefore,
$$
\varepsilon\int \frac{\varphi'(f)}{\langle\nabla f,\nabla f\rangle+\varepsilon}
\, d\gamma_n \le
 \\
\|\varphi'\|_\infty c(d)\sigma_f^{-2\beta}\varepsilon^{\beta}.
$$
Let
$$
\tau=\frac{q-1}{q}, \quad
\varepsilon = \|\varphi'\|_\infty^\omega,
\quad  \omega={-\frac{1}{1+\beta\tau}}.
$$
Then  for (\ref{ek5.1}) we have the bound
$$
\int \varphi'(f(x))\, \gamma_n(dx)\le (c_1(d,q)\sigma_f^{1-2\beta/q}
+ c(d)\sigma_f^{-2\beta})\|\varphi'\|_\infty^{1-\alpha},
\quad \alpha={\frac{1}{2d-2+\tau}}.
$$
We now take the function $\psi(t) = \varphi(t\sigma_f^{-1})$.
Using the above inequality for the polynomial $f\cdot\sigma_f^{-1}$, we can write
\begin{multline*}
\int \psi'(f(x))\, \gamma_n(dx) = \sigma_f^{-1}\int \varphi'(f(x)\sigma_f^{-1})\, \gamma_n(dx)
\\
\le \sigma_f^{-1}(c_1(d,q) + c(d))\|\varphi'\|_\infty^{1-\alpha}=
\sigma_f^{-\alpha}(c_1(d,q) + c(d))\|\psi'\|_\infty^{1-\alpha}.
\end{multline*}
Since $\tau$ can be taken as small as we wish, the theorem is proved.
\end{proof}

The last assertion about membership in Nikol'skii--Besov  classes is improved below.
Similarly to the multidimensional case, the following theorem is obtained
on the basis of the previous theorem.

\begin{theorem}\label{t5.1b}
Let $d\in \mathbb{N}$, $a>0$, $\tau>0$. Then there is a number
$C=C(d, a, \tau)>0$ such that, whenever $f$ and $g$ are
real polynomials on $\mathbb{R}^n$ of degree~$d$ with
$\sigma_{f}, \sigma_g\ge a$
one has
$$
\|\gamma_n\circ f^{-1} - \gamma_n\circ g^{-1}\|_{{\rm TV}} \le
C d_{{\rm K}} (\gamma_n\circ f^{-1}, \gamma_n\circ g^{-1})^\theta, \quad \theta=\frac{1}{2d-1+\tau}.
$$
\end{theorem}

As in the multidimensional case, we obtain the following
infinite-dimensional extensions.

\begin{corollary}\label{newcor5.1}
Let $\gamma$ be a centered Radon Gaussian measure on a locally convex space~$X$.
Let $d\in \mathbb{N}$, $\tau>0$. Then there is a number
$C(d,\tau)>0$ such that, whenever
$f\colon\, X\to\mathbb{R}$ is a $\gamma$-measurable polynomial of degree~$d$,
for all $\varphi\in C_b^\infty(\mathbb{R}^1)$ one has
$$
\int_{X}
 \varphi'(f(x))\, \gamma(dx)\le C(d,\tau) \sigma_f^{-\alpha}
 \|\varphi\|_\infty^{\alpha} \|\varphi'\|_{\infty}^{1-\alpha},
\quad \alpha=\frac{1}{2d-2+\tau}.
$$
Therefore, $\gamma\circ f^{-1}$ belongs
to the Nikol'skii--Besov  class~$B^\alpha(\mathbb{R})$, provided that $f$ is not a
constant a.e.
\end{corollary}

\begin{corollary}\label{c5.1}
Let $\gamma$ be a centered Radon Gaussian measure on a locally convex space~$X$.
Let $d\in \mathbb{N}$, $a>0$, $\tau>0$. Then is a number
$C_1=C_1(d, a, \tau)$ such that, whenever $f$ and $g$ are
$\gamma$-measurable polynomials on $X$ of degree~$d$ with
$\sigma_{f}, \sigma_g\ge a$
one has
$$
\|\gamma\circ f^{-1} - \gamma\circ g^{-1}\|_{{\rm TV}} \le
C_1 d_{{\rm FM}} (\gamma\circ f^{-1}, \gamma\circ g^{-1})^{1/(2d-1+\tau)}.
$$
\end{corollary}

\begin{corollary}\label{c5.2}
Let $\gamma$ be a  Radon Gaussian measure on a locally convex space.
Let $\{f_n\}$ be a sequence of $\gamma$-measurable polynomials of degree~$d$.
Suppose that the distributions $\gamma \circ f_n^{-1}$ converge weakly
to an absolutely continuous measure $\nu$ on~$\mathbb R$.
Then they also converge in variation
and for every $\tau>0$ there exists a number
$C_2=C_2(d,\sigma_\nu, \tau)$ such that
$$
d_{{\rm TV}} (\gamma\circ f_m^{-1}, \gamma\circ f_n^{-1})\le
C_2 d_{{\rm FM}} (\gamma\circ f_m^{-1}, \gamma\circ f_n^{-1})^{1/(2d-1+\tau)}.
$$
\end{corollary}

The second result provides an estimate with a better rate of convergence
than the one obtained in Theorem 3.1 in~\cite{NP}.

\begin{remark}
\rm
Note that in this case, unlike Corollary \ref{c4.2}, there is no condition
that the integrals of $\Delta_{f_n}$ are separated from zero.
In the case $k=1$, due to the Poincar\'e inequality,
this condition is replaced by $\sigma_{f_n}\ge a>0$
(see Corollary \ref{c4.1} and Corollary \ref{c5.1}),
which is automatically satisfied for $n$ large enough, because for
the distributions of polynomials
weak convergence implies convergence of all moments (see Lemma \ref{moments-weak}).
\end{remark}

We now show that one can even achieve the exponent $\theta=1/(d+1)$, however,
with a worse constant than before (depending on some special norm of the gradient).
Actually, by using a different approach
in the one-dimensional case, it is still possible to prove this result
with the same type of constant (depending on the variance), which will be done for general
convex measures in a forthcoming paper of the second author. We include a somewhat less sharp
result below, because its proof is much simpler.

Let $\gamma$ be a centered Radon Gaussian measure on a locally
convex space $X$ and let $H$ be its Cameron--Martin space.
For a function $f\in W^{2,1}(\gamma)$ we define $\|\nabla f\|_{*}$ by
\begin{equation}\label{norm}
\|\nabla f\|_{*}^2:= \sup\limits_{|e|_H=1}\int_X |\partial_e f|^2 d\gamma.
\end{equation}
It is clear that $\|\nabla f\|_{*}>0$ once $f$ is not a constant and that
$\|\nabla f\|_{*}\le \|\, |\nabla f|_H\, \|_{L^2(\gamma)}$.

\begin{theorem}\label{t-new}
Let $\gamma_n$ be the standard Gaussian measure on $\mathbb{R}^n$.
Then, for every $d\in\mathbb{N}$,
 there is a number $C(d)$ that depends only on~$d$ such that,
for every polynomial $f$ of degree $d$ on $\mathbb{R}^n$
 and every function $\varphi\in C_b^\infty(\mathbb{R})$, we have
$$
\int_{\mathbb{R}^n}
\varphi'(f)\, d\gamma_n\le C(d) \|\nabla f\|_{*}^{-1/d}
\|\varphi\|_\infty^{1/d} \|\varphi'\|_\infty^{1-1/d}.
$$
Therefore, $\gamma_n\circ f^{-1}$ belongs to the Nikol'skii--Besov  class
$B^{1/d}(\mathbb{R})$ provided that $f$ is not a constant.
\end{theorem}
\begin{proof}
We can assume that $\|\varphi\|_\infty\le 1$.
Let $e\in \mathbb{R}^n$, $|e|=1$.
We have
$$
\int\varphi'(f)\, d\gamma
=\int\Bigl[\frac{(\partial_ef)^2}{(\partial_ef)^2+\varepsilon}\varphi'(f)\Bigr]\, d\gamma
+ \varepsilon\int \frac{\varphi'(f)}{(\partial_ef)^2+\varepsilon}\, d\gamma.
$$
Writing the first term as
$$
\int\frac{(\partial_ef)^2}{(\partial_ef)^2+\varepsilon}\varphi'(f)\, d\gamma=
\int\partial_e(\varphi(f))\frac{\partial_ef}{(\partial_ef)^2+\varepsilon}\, d\gamma
$$
and integrating by parts in the last expression, we obtain
\begin{align*}
-\int\varphi(f) & \Bigl[\frac{\partial^2_ef + \langle x, e\rangle\partial_ef}{(\partial_ef)^2+\varepsilon}-
2\frac{(\partial_ef)^2\partial^2_ef}{((\partial_ef)^2+\varepsilon)^2}\Bigr]\, d\gamma
\\
&\le
3\int\Bigl|\frac{\partial^2_ef}{(\partial_ef)^2+\varepsilon}\Bigr|\, d\gamma+
\int\Bigl|\frac{\partial_ef}{(\partial_ef)^2+\varepsilon}\Bigr|\, |\langle x, e\rangle|
\, d\gamma
\\
&=\varepsilon^{-1/2}
\biggl(3\int\Bigl|\frac{\partial^2_eg}{(\partial_eg)^2+1}\Bigr|d\gamma+
\int\Bigl|\frac{\partial_eg}{(\partial_eg)^2+1}\Bigr||\langle x, e\rangle|\,
d\gamma\biggr)
\\
&\le \varepsilon^{-1/2}(3d\sqrt{\pi/2}+1),
\end{align*}
where $g=f\varepsilon^{-1/2}$.
By using the Carbery--Wright inequality (\ref{cw-ineq}) in the same manner
as in derivation of (\ref{CaW4.4}) we have
$$
\int \frac{\varphi'(f)}{(\partial_ef)^2+\varepsilon}\, d\gamma
\le cd \|\varphi'\|_\infty
 \|\partial_ef\|_2^{-1/(d-1)}\varepsilon^{-1+1/(2d-2)}
\int_0^\infty(s+1)^{-2}s^{1/(2d-2)}\, ds.
$$
Thus,
$$
\int\varphi'(f)d\gamma\le c_1(d)\|\partial_ef\|_2^{-1/(d-1)}
\|\varphi'\|_\infty\varepsilon^{1/(2d-2)}+c_2(d)\varepsilon^{-1/2}.
$$
Taking $\varepsilon=\|\varphi'\|_\infty^{-2+2/d}$, we obtain
$$
\int\varphi'(f)\, d\gamma\le(c_1(d)\|\partial_ef\|_2^{-1/(d-1)}+c_2(d))\|\varphi'\|_\infty^{1-1/d}.
$$
Since this estimate is valid for every vector $e\in \mathbb{R}^n$ of unit length, we have
$$
\int\varphi'(f)\, d\gamma\le
\Bigl(c_1(d)\|\nabla f\|_{*}^{-1/(d-1)}+c_2(d)\Bigr)\|\varphi'\|_\infty^{1-1/d}.
$$
Applying the last estimate to the polynomial $f\|\nabla f\|_{*}^{-1}$, we find that
$$
\int\varphi'(f\|\nabla f\|_{*}^{-1})\, d\gamma\le
(c_1(d)+c_2(d))\|\varphi'\|_\infty^{1-1/d}.
$$
Let $\psi(t)=\varphi(t\|\nabla f\|_{*}^{-1})$, $C(d)=c_1(d)+c_2(d)$. Then
\begin{align*}
\int\psi'(f)\, d\gamma
&=\|\nabla f\|_{*}^{-1}\int\varphi'(f\|\nabla f\|_{*}^{-1})\, d\gamma
\\
&
\le
C(d)\|\nabla f\|_{*}^{-1}\|\varphi'\|_\infty^{1-1/d}=
C(d)\|\nabla f\|_{*}^{-1/d}\|\psi'\|_\infty^{1-1/d},
\end{align*}
which proves the theorem.
\end{proof}

\begin{corollary}\label{ic-new}
Let $\gamma$ be a centered Radon Gaussian measure on a locally convex space.
Then, for every $d\in\mathbb{N}$,
 there is a number $C(d)$ that depends only on~$d$ such that,
for every $\gamma$-measurable polynomial $f$ of degree $d$ on $X$
 and every function $\varphi\in C_b^\infty(\mathbb{R})$, we have
$$
\int_{X}
\varphi'(f)\, d\gamma\le C(d) \|\nabla f\|_{*}^{-1/d}
\|\varphi\|_\infty^{1/d} \|\varphi'\|_\infty^{1-1/d}.
$$
Therefore, $\gamma\circ f^{-1}$ belongs to the Nikol'skii--Besov  class
$B^{1/d}(\mathbb{R})$ provided that $f$ is not a constant a.e.
\end{corollary}

From the previous theorem one derives
the following assertion which is an analog of Theorem \ref{t5.1b} in this case.

\begin{theorem}\label{t5.2}
Let $d\in \mathbb{N}$, $a>0$. Then there is a number
$C=C(d, a)$ such that, whenever $f$ and $g$ are
real polynomials on $\mathbb{R}^n$ of degree~$d$ with
$\|\nabla f\|_{*}\ge a$ and $\|\nabla g\|_{*}\ge a$,
one has
$$
\|\gamma_n\circ f^{-1} - \gamma_n\circ g^{-1}\|_{{\rm TV}} \le
C d_{{\rm K}} (\gamma_n\circ f^{-1}, \gamma_n\circ g^{-1})^\theta, \quad \theta=\frac{1}{d+1}.
$$
\end{theorem}

It can be that the optimal power is $1/d$; the following simple example shows that one cannot
get any better exponent (and that the order $1/d$ of the
Nikol'skii--Besov class above is optimal).

\begin{example}
\rm
Let us consider the monomial $x^{d}$ with even $d$
on the real line with the standard Gaussian measure~$\gamma$.
Let $\varrho$ be its distribution density.
It is obvious that $\varrho(t)=0$ if $t<0$ and that
$\varrho$ is monotonically decreasing on $(0, +\infty)$.
Let us also consider $x^{d}-h$, $h>0$.
The Kantorovich distance between the laws of $x^{d}$ and $x^{d}-h$ equals $h$
and the variation distance is given by
$$
\int_{-\infty}^{+\infty}
 |\varrho(t-h) - \varrho(t)|\, dt = \int_0^h \varrho(t)\, dt
+ \int_h^{+\infty}(\varrho(t-h) - \varrho(t))\, dt=
2\gamma(|x|\le h^{1/d}).
$$
It is readily verified that the latter expression
for small $h$ behaves like~$h^{1/d}$.
\end{example}

\begin{remark}
\rm
It is still
unknown whether the set of distributions of polynomials of a fixed degree~$d$
is closed in the weak topology (equivalently, in the metrics $d_{{\rm K}}$
and~$d_{{\rm FM}}$). The answer is positive for $d=1$ (which is trivial)
and for $d=2$ (which was proved in \cite{Arcones} and~\cite{Sev}).
Some asymptotic properties of polynomial distributions are discussed in
\cite{AgB} and~\cite{B14}.
\end{remark}

\section{Bounds via $L^2$-norms}

In this section, $\gamma$ is the standard Gaussian measure on~$\mathbb{R}^n$
(in this case we also use the symbol~$\gamma_n$) or on~$\mathbb{R}^\infty$.
The following result was announced  in \cite{DavM}
(we present it in our terms; in \cite{DavM} multiple stochastic
integrals of order $d$ are used).

{\bf Theorem A.} Let $g\in\mathcal{H}_d$ and $g\ne0$. Then there is a constant $C(d, g)$
depending only on $d$ and $g$ such that for every $f\in\mathcal{H}_d$ one has
$$
\|\gamma\circ f^{-1} - \gamma\circ g^{-1}\|_{\rm TV}\le
C(d, g) \|f-g\|_2^{1/d}.
$$

The announcement does not contain details of proof and also
the form of dependence of  $C(d,g)$ on $g$ is not indicated.
In relation to this estimate   Nourdin and Poly \cite{NP} proved the following result
(also presented here in our terms).

{\bf Theorem B.} Let $d\in\mathbb{N}$, $a>0$, $b> 0$. Then there exists a number $C(d, a, b)>0$
such that for every pair of polynomials $f, g$ of degree~$d$ with $\sigma_f\in [a,b]$
one has
$$
\|\gamma\circ f^{-1} - \gamma\circ g^{-1}\|_{\rm TV}
\le C(d, a, b)\|f-g\|_2^{1/(2d)}.
$$

While the power of the $L^2$-norm in this theorem is twice smaller
(which makes the estimate worse) than in Theorem A,
Nourdin and Poly managed to clarify dependence of $C(d, g)$ on~$g$:
this constant depends only on the bounds for the variance.
In this section, we first prove an intermediate result between Theorem A and Theorem B
and then give its multidimensional extension.
The next theorem gives an analog of the Davydov--Martynova estimate
with a constant worse than in the Nourdin--Poly estimate, but with a better dependence on the $L^2$-norm
(which differs from the announcement in \cite{DavM}  by only a logarithmic factor).
We recall that $\|\cdot\|_*$  is defined by (\ref{norm}).

\begin{theorem}
There is a constant $c(d)$ depending only on $d$ such that
for every pair of polynomials $f, g$ of degree $d>1$ one has
$$
\|\gamma\circ f^{-1} - \gamma\circ g^{-1}\|_{\rm TV}\le
c(d)\bigl(\|\nabla g\|_*^{-1/(d-1)} +
\sigma_g + 1\bigr)\|f-g\|_2^{1/d}\Bigl(\bigl|\ln\|f-g\|_2\bigr|^{d/2}+1\Bigr).
$$
\end{theorem}
\begin{proof}
If $\|f-g\|_2\ge1/e$, then
$$
\|\gamma\circ f^{-1} - \gamma\circ g^{-1}\|_{\rm TV}\le1\le
e^{1/d}\|f-g\|^{1/d}_2\Bigl(\bigl|\ln\|f-g\|_2\bigr|^{d/2}+1\Bigr).
$$
Hence we can assume that $\|f-g\|_2\le1/e$.
Fix a function $\varphi\in C_0^\infty(\mathbb{R})$ with $\|\varphi\|_\infty\le1$,
a vector $e\in\mathbb{R}^n$ of unit length, and a number $\varepsilon\in (0, 1/e)$.
Consider the function
$$
\Phi(t) = \int_{-\infty}^t\varphi(\tau)d\tau
$$
Note that
$$
\partial_e(\Phi(f) - \Phi(g)) = \partial_ef\varphi(f)-\partial_eg\varphi(g)=
(\varphi(f)-\varphi(g))\partial_eg+\varphi(f)(\partial_ef-\partial_eg).
$$
Thus, we have (omitting indication of limits of integration in case of~$\mathbb{R}^n$)
\begin{multline*}
\int\varphi(f)-\varphi(g)d\gamma=
\int(\varphi(f)-\varphi(g))\frac{(\partial_eg)^2}{(\partial_eg)^2+\varepsilon}d\gamma
\\
+\varepsilon\int(\varphi(f)-\varphi(g))((\partial_eg)^2+\varepsilon)^{-1}d\gamma
\\
=\int\frac{\partial_eg\partial_e(\Phi(f) - \Phi(g))}{(\partial_eg)^2+\varepsilon}d\gamma -
\int\frac{\varphi(f)(\partial_ef-\partial_eg)\partial_eg}{(\partial_eg)^2+\varepsilon}d\gamma
\\
+\varepsilon\int(\varphi(f)-\varphi(g))((\partial_eg)^2+\varepsilon)^{-1}d\gamma.
\end{multline*}
Let us estimate each term separately.
First, let us consider the last term.
Using the Carbery--Wright inequality in the same manner
as in derivation of (\ref{CaW4.4}) we obtain
\begin{multline*}
\varepsilon\int(\varphi(f)-\varphi(g))((\partial_eg)^2+\varepsilon)^{-1}d\gamma\le
2\varepsilon\int((\partial_eg)^2+\varepsilon)^{-1}d\gamma
\\
\le 2dc_1 \|\partial_eg\|_2^{-1/(d-1)}\varepsilon^{1/(2d-2)}
\int_0^\infty(s+1)^{-2}s^{1/(2d-2)}ds
\\
=
c_1(d)\|\partial_eg\|_2^{-1/(d-1)}\varepsilon^{1/(2d-2)}.
\end{multline*}
Now we estimate the second term:
\begin{multline*}
-\int\frac{\varphi(f)(\partial_ef-\partial_eg)\partial_eg}{(\partial_eg)^2+\varepsilon}d\gamma\le
\int\frac{|\partial_ef-\partial_eg|\, |\partial_eg|}{(\partial_eg)^2+\varepsilon}d\gamma
\\
\le2^{-1}\varepsilon^{-1/2}\int|\partial_ef-\partial_eg|d\gamma \le c_2(d)\varepsilon^{-1/2}\|f-g\|_2.
\end{multline*}
Finally, let us estimate the first term. Integrating by parts we obtain
\begin{multline*}
\int\frac{\partial_eg\partial_e(\Phi(f) - \Phi(g))}{(\partial_eg)^2+\varepsilon}d\gamma
\\
=-\int(\Phi(f) - \Phi(g))\Bigl[\frac{\partial^2_eg-\langle x, e\rangle\partial_eg}{(\partial_eg)^2+\varepsilon}-
2\frac{(\partial_eg)^2\partial^2_eg}{((\partial_eg)^2+\varepsilon)^2}\Bigr]d\gamma
\\
\le3\int|f-g|\frac{|\partial^2_eg|}{(\partial_eg)^2+\varepsilon}d\gamma
+2^{-1}\varepsilon^{-1/2}\int|f-g|\, |\langle x, e\rangle|d\gamma
\\
\le3\int_{\{|f-g|\ge \|f-g\|_2t\}}|f-g|\frac{|\partial^2_eg|}{(\partial_eg)^2+\varepsilon}d\gamma+
3t\|f-g\|_2\int\frac{|\partial^2_eg|}{(\partial_eg)^2+\varepsilon}d\gamma
\\
+2^{-1}\varepsilon^{-1/2}\|f-g\|_2\le
c_3(d)\varepsilon^{-1}\|f-g\|_2\sigma_g\bigl(\gamma(|f-g|\ge \|f-g\|_2t)\bigr)^{1/3}
\\
+3t\|f-g\|_2\int\frac{|\partial^2_eg|}{(\partial_eg)^2+\varepsilon}d\gamma
+2^{-1}\varepsilon^{-1/2}\|f-g\|_2.
\end{multline*}
Note that writing $\gamma$ as the product of $\gamma_1$ and $\gamma_{n-1}$,
we have
\begin{multline*}
\int\frac{|\partial^2_eg|}{(\partial_eg)^2+\varepsilon} d\gamma =
\int_{\langle e\rangle^\bot}\int_{\langle e \rangle}
\frac{|\partial^2_eg|}{(\partial_eg)^2+\varepsilon} d\gamma_1\, d\gamma_{n-1}
\\
=\varepsilon^{-1/2}\int_{\langle e\rangle^\bot}\int_{\langle e \rangle}
\frac{|\partial^2_eg\varepsilon^{-1/2}|}{(\partial_eg\varepsilon^{-1/2})^2+1} d\gamma_1\, d\gamma_{n-1}
\le d (2\pi \varepsilon)^{-1/2}\int_{\langle e\rangle^\bot}\int\frac{d\tau}{\tau^2+1}\, d\gamma_{n-1}
\\
=d\varepsilon^{-1/2}(2\pi)^{-1/2}\int\frac{1}{\tau^2+1}d\tau = 3^{-1}c_4(d)\varepsilon^{-1/2}.
\end{multline*}
Recall (see \cite[Corollary~5.5.7]{GM}) that
$$
\gamma\bigl(x\colon \ |f(x)|\ge t \|f\|_2\bigr)\le c_r \exp (- r t^{2/d}), \quad
r<\frac{d}{2e},
$$
where $c_r$ depends only on~$r$.
Thus, for $t\ge 1$ and some $c\in (0,1/2)$ we obtain
\begin{multline*}
\int [\varphi(f)-\varphi(g)] d\gamma
\le
c_5(d)\Bigl(\|\partial_eg\|_2^{-1/(d-1)}\varepsilon^{1/(2d-2)}
\\
 +
\varepsilon^{-1}\|f-g\|_2\sigma_g
\exp\bigl(-ct^{2/d}\bigr) + t\|f-g\|_2\varepsilon^{-1/2}\Bigr).
\end{multline*}
Setting $t = (2c)^{-d/2}(\ln\varepsilon^{-1})^{d/2}$,
$\varepsilon = \|f-g\|_{2}^{2(d-1)/d}$
(recall that $\|f-g\|_{2}<1/e$, hence $t\ge 1$),
we obtain that the right-hand side is estimated by
\begin{multline*}
c(d)\Bigl(\|\partial_eg\|_2^{-1/(d-1)}\|f-g\|_2^{1/d} +
\sigma_g\|f-g\|_2^{1/d} + \bigl|\ln\|f-g\|_2\bigr|^{d/2}\|f-g\|_2^{1/d}\Bigr)
\\
\le c(d)\bigl(\|\partial_eg\|_2^{-1/(d-1)} +
\sigma_g + 1\bigr)\bigl|\ln\|f-g\|_2\bigr|^{d/2}\|f-g\|_2^{1/d}.
\end{multline*}
Now taking inf over $e$ and sup over $\varphi$ we obtain the desired estimate.
\end{proof}

Our next theorem is a multidimensional analog of Theorem A.
We need a lemma.

\begin{lemma}\label{l1}
Let $A$ and $B$ be a pair of square $k\times k$-matrices. Then
$$
|\det A - \det B| \le\|A-B\|_{HS}\bigl(\|A\|_{HS}^2 + \|B\|_{HS}^2\bigr)^{(k-1)/2}.
$$
\end{lemma}
\begin{proof}
Let $a_i$ and $b_i$, $i=1,\ldots,k$,  be the columns of the
matrices $A$ and $B$, respectively. The determinant of the matrix $A$
is a multilinear function in $a_1,\ldots, a_k$. We denote this
function by $\Delta(a_1,\ldots, a_k)$. We have
\begin{multline*}
|\det A - \det B| = |\Delta(a_1,\ldots, a_k) - \Delta(b_1,\ldots, b_k)|
\\
\le
\sum_{i=1}^k|\Delta(b_1,\ldots, b_{i-1}, a_i,\ldots, a_k) - \Delta(b_1,\ldots, b_i, a_{i+1},\ldots a_k)|
\\
=\sum_{i=1}^k|\Delta(b_1,\ldots, b_{i-1}, a_i - b_i, a_{i+1}, \ldots, a_k)|\le
\sum_{i=1}^k |b_1|\ldots|b_{i-1}||a_i - b_i||a_{i+1}|\ldots|a_k|
\\
\le
\Bigl(\sum_{i=1}^k|a_i - b_i|^2\Bigr)^{1/2}
\Bigl(\sum_{i=1}^k (|a_i|^2 + |b_i|^2)\Bigr)^{(k-1)/2}
\\
=
\|A-B\|_{HS}(\|A\|_{HS}^2 + \|B\|_{HS}^2)^{(k-1)/2}.
\end{multline*}
The lemma is proved.
\end{proof}

\begin{theorem}\label{t6.3}
Let $k,d\in\mathbb{N}$, $a>0$, $b> 0$, $\tau>0$. Then there exists a number  $C(d, k, a, b, \tau)>0$
such that, for every pair of mappings $f = (f_1, \ldots, f_k)$ and $g~=~(g_1, \ldots, g_k)\colon\,
\mathbb{R}^n\to\mathbb{R}^k$, where all $f_i, g_i$ are polynomials of degree~$d$ and
$$
\int_{\mathbb{R}^n} \Delta_f \, d\gamma\ge a,
\quad \max_{i\le k} \sigma_{f_i}\le b,
$$
one has
\begin{equation}\label{est1}
\|\gamma\circ f^{-1} - \gamma\circ g^{-1}\|_{\rm TV}
\le C(d, k, a, b, \tau)\|f-g\|_2^{\theta}, \ \theta=\frac{1}{4k(d-1)+\tau}.
\end{equation}
\end{theorem}
\begin{proof}
Fix $\varphi\in C_0^\infty(\mathbb{R}^k)$ with $\|\varphi\|_\infty\le1$.
Let $f^i = (g_1, \ldots, g_i, f_{i+1},\ldots, f_k)$, $f^0 = f$, $f^k = g$.
Consider the function
$$
\Phi_i(y_1,\ldots, y_k) =
\int_{-\infty}^{y_i} \varphi(y_1,\ldots, y_{i-1}, t, y_{i+1},\ldots, y_k)dt.
$$
Note that for each $i$ we have
$$
\nabla (\Phi_i(f^{i-1})) - \nabla (\Phi_i(f^i) ) =
\sum_{j=1}^k
\bigl(\partial_{y_j}\Phi_i(f^{i-1})\nabla f^{i-1}_j -
\partial_{y_j}\Phi_i(f^i)\nabla f^i_j\bigr),
$$
which can be written as
\begin{multline*}
\sum_{j=1}^k \bigl(\partial_{y_j}\Phi_i(f^{i-1})
- \partial_{y_j}\Phi_i(f^i)\bigr)\nabla f^{i-1}_j
+ \partial_{y_i}\Phi_i(f^i)(\nabla f^{i-1}_i - \nabla f^i_i)
\\
=\sum_{j=1}^k \bigl(\partial_{y_j}\Phi_i(f^{i-1})
- \partial_{y_j}\Phi_i(f^i)\bigr)\nabla f^{i-1}_j + \varphi(f^i)(\nabla f_i - \nabla g_i).
\end{multline*}
Thus,
\begin{multline*}
\bigl(\langle\nabla\Phi_i(f^{i-1}) - \nabla\Phi_i(f^i), \nabla f^{i-1}_m\rangle\bigr)_{m=1}^k
 \\
=M_{f^{i-1}}
\bigl(\partial_{y_j}\Phi_i(f^{i-1}) - \partial_{y_j}\Phi_i(f^i)\bigr)_{j=1}^k
+  \varphi(f^i)\bigl(\langle\nabla f_i - \nabla g_i, \nabla f^{i-1}_m\rangle\bigr)_{m=1}^k.
\end{multline*}
Recall that
$\Delta_f\cdot M_f^{-1} = A_f$ (see (\ref{inver})).
Hence,
denoting the elements of the matrix $A_f^k$ by $a^{s,r}_{f^k}$,
 we obtain
\begin{multline}\label{f1}
\Delta_{f^{i-1}}(\varphi(f^{i-1})
- \varphi(f^i)) = \Delta_{f^{i-1}}(\partial_{y_i}\Phi_i(f^{i-1}) - \partial_{y_i}\Phi_i(f^i))
\\
=\sum_{j=1}^k \langle\nabla\Phi_i(f^{i-1}) - \nabla\Phi_i(f^i), \nabla f^{i-1}_j\rangle a^{j,i}_{f^{i-1}} -
\varphi(f^i)\sum_{j=1}^k \langle\nabla f_i - \nabla g_i, \nabla f^{i-1}_j\rangle a^{j,i}_{f^{i-1}}.
\end{multline}
Next we observe that
\begin{multline}\label{f2}
\int [\varphi(f) - \varphi(g)] d\gamma
\\
= \sum_{i=1}^{k}\int \frac{\Delta_{f^{i-1}}(\varphi(f^{i-1})
- \varphi(f^i))}{\Delta_f + \varepsilon}d\gamma
 +
\sum_{i=1}^{k}\int \frac{(\Delta_{f^{i-1}} - \Delta_{f^i})\varphi(f^i)}{\Delta_f + \varepsilon}d\gamma
\\
+\int\frac{(\Delta_g - \Delta_f)\varphi(g)}
{\Delta_f + \varepsilon}d\gamma + \int\varepsilon(\varphi(f) - \varphi(g))(\Delta_f + \varepsilon)^{-1}d\gamma.
\end{multline}
Let us estimate each term separately.
Let
$$
\beta=(2k(d-1))^{-1}.
$$
Recall (see (\ref{CaW4.4})) that
$$
\int (\Delta_f+\varepsilon)^{-p}\, d\gamma \le
c(p,d)^p\varepsilon^{-p+\beta}\biggl(\int \Delta_f\, d\gamma\biggr)^{-\beta}.
$$
Using this inequality, we estimate the last term in the right-hand side of (\ref{f2}):
$$
\int\varepsilon(\varphi(f) - \varphi(g))
(\Delta_f + \varepsilon)^{-1}d\gamma \le 2c(1,d)\varepsilon^\beta
\biggl(\int \Delta_f\, d\gamma\biggr)^{-\beta}.
$$
The second and the third term in (\ref{f2}) can be estimated as follows.
By Lemma \ref{l1} we have
\begin{multline*}
|\Delta_g - \Delta_f|\le \Bigl(\|M_f\|_{HS}^2 + \|M_g\|_{HS}^2\Bigr)^{(k-1)/2}\|M_f - M_g\|_{HS}
\\
\le
\sqrt{2}\Bigl(\sum_{i=1}^k \bigl(|\nabla f_i|^2 + |\nabla g_i|^2\bigr)\Bigr)^{k-1/2}
\Bigl(\sum_{i=1}^k |\nabla f_i-\nabla g_i|^2\Bigr)^{1/2},
\end{multline*}
where we used the estimates
$\|M_f\|_{HS}^2=\sum_{i,j} \langle\nabla f_i, \nabla f_j\rangle^2 \le
\Bigl(\sum_{i} |\nabla f_i|^2\Bigr)^2$
and
\begin{multline*}
\|M_f - M_g\|_{HS}^2 = \sum_{i,j}
\bigl(\langle\nabla f_i, \nabla f_j\rangle - \langle\nabla g_i, \nabla g_j\rangle\bigr)^2
\\
 \le
2 \sum_{i,j} \bigl[ \bigl(\langle\nabla f_i, \nabla f_j\rangle
- \langle\nabla f_i, \nabla g_j\rangle\bigr)^2
+ \bigl(\langle\nabla f_i, \nabla g_j\rangle
 - \langle\nabla g_i, \nabla g_j\rangle\bigr)^2\bigr]
\\
\le
2 \sum_{i,j} \bigl[ |\nabla f_i|^2 |\nabla f_j-\nabla g_j|^2
+ |\nabla f_i  - \nabla g_i|^2 |\nabla g_j|^2\bigr]
\\
=
2 \sum_{i}
|\nabla f_i - \nabla g_i|^2 \sum_{i} \bigl(|\nabla f_i|^2 + |\nabla g_i|^2\bigr).
\end{multline*}
Similarly,
$$
|\Delta_{f^{i-1}} - \Delta_{f^i}|\le 2^k
\Bigl(\sum_{i=1}^k
(|\nabla f_i|^2 + |\nabla g_i|^2)\Bigr)^{k-1/2}
\Bigl(\sum_{i=1}^k |\nabla f_i-\nabla g_i|^2\Bigr)^{1/2}.
$$
Using these estimates we obtain
\begin{multline*}
\int \frac{(\Delta_{f^{i-1}} - \Delta_{f^i})\varphi(f^i)}{\Delta_f + \varepsilon}d\gamma \le
\int \frac{|\Delta_{f^{i-1}} - \Delta_{f^i}|}{\Delta_f + \varepsilon}d\gamma
\\
\le
2^k\int \Bigl(\sum_{i=1}^k (|\nabla f_i|^2
+ |\nabla g_i|^2)\Bigr)^{k-1/2}
\Bigl(\sum_{i=1}^k |\nabla f_i-\nabla g_i|^2\Bigr)^{1/2}
(\Delta_f + \varepsilon)^{-1}d\gamma
\\
\le
C(k, d) \Bigl(\sum_{i=1}^k (\sigma^2_{f_i} + \sigma_{g_i}^2)\Bigr)^{k-1/2}
\|f-g\|_2\varepsilon^{-1}.
\end{multline*}
Similarly,
$$
\int \frac{(\Delta_{g} - \Delta_{f})\varphi(g)}{\Delta_f + \varepsilon}d\gamma\le
C(k, d) \Bigl(\sum_{i=1}^k
(\sigma^2_{f_i} + \sigma_{g_i}^2)\Bigr)^{k-1/2}\|f-g\|_2\varepsilon^{-1}.
$$
Let us now consider the first term in the right-hand side of (\ref{f2}). By (\ref{f1}) we have
\begin{multline}\label{f3}
\int \frac{\Delta_{f^{i-1}}(\varphi(f^{i-1}) - \varphi(f^i))}{\Delta_f + \varepsilon}d\gamma
\\
=\int (\Delta_f + \varepsilon)^{-1}\sum_{j=1}^k \langle\nabla\Phi_i(f^{i-1})
- \nabla\Phi_i(f^i), \nabla f^{i-1}_j\rangle a^{j,i}_{f^{i-1}}d\gamma
\\
-\int\varphi(f^i)(\Delta_f + \varepsilon)^{-1}
\sum_{j=1}^k \langle\nabla f_i-\nabla g_i, \nabla f^{i-1}_j\rangle a^{j,i}_{f^{i-1}}d\gamma.
\end{multline}
The second term in (\ref{f3}) can be estimated in the following way:
\begin{multline*}
\int\varphi(f^i)(\Delta_f + \varepsilon)^{-1}\langle\nabla f_i-\nabla g_i,
\nabla f^{i-1}_j\rangle a^{j,i}_{f^{i-1}}d\gamma
\\
\le
\int(\Delta_f + \varepsilon)^{-1}|\nabla f_i-\nabla g_i| \,
|\nabla f^{i-1}_j|\, |a^{j,i}_{f^{i-1}}|d\gamma
\\
\le
\varepsilon^{-1}(k-1)^2!\int\Bigl(\sum_{i=1}^k (|\nabla f_i|^2 + |\nabla g_i|^2)
\Bigr)^{k-1/2}
\Bigl(\sum_{i=1}^k |\nabla f_i-\nabla g_i|^2\Bigr)^{1/2}d\gamma
\\
\le C(k, d)\varepsilon^{-1} \Bigl(\sum_{i=1}^k (\sigma^2_{f_i} + \sigma_{g_i}^2)
\Bigr)^{k-1/2}\|f-g\|_2.
\end{multline*}
Finally, let us consider
the first term in (\ref{f3}). Fix $p>1$. Integrating by parts we have
\begin{multline*}
\int (\Delta_f + \varepsilon)^{-1}\langle\nabla\Phi_i(f^{i-1})
- \nabla\Phi_i(f^i), \nabla f^{i-1}_j\rangle a^{j,i}_{f^{i-1}}d\gamma
= - \int (\Phi_i(f^{i-1}) - \Phi_i(f^i))
\\
\times \Bigl(\frac{a^{j,i}_{f^{i-1}}
L f^{i-1}_j}{\Delta_f + \varepsilon} - \frac{a^{j,i}_{f^{i-1}}
\langle\nabla f^{i-1}_j, \nabla \Delta_f \rangle}
{(\Delta_f + \varepsilon)^{2}}
+ \frac{\langle\nabla f^{i-1}_j, \nabla a^{j,i}_{f^{i-1}}\rangle}{\Delta_f
+ \varepsilon}\Bigr)d\gamma
\\
\le \int |f_i - g_i|\Bigl(\frac{|a^{j,i}_{f^{i-1}}L f^{i-1}_j|}
{\Delta_f + \varepsilon} + \frac{|a^{j,i}_{f^{i-1}} \langle\nabla f^{i-1}_j, \nabla \Delta_f \rangle|}
{(\Delta_f + \varepsilon)^{2}}
+ \frac{|\langle\nabla f^{i-1}_j, \nabla a^{j,i}_{f^{i-1}}\rangle|}
{\Delta_f + \varepsilon}\Bigr)d\gamma,
\end{multline*}
which is estimated by
\begin{multline*}
\varepsilon^{-1}\|f-g\|_2\Bigl(\|a^{j,i}_{f^{i-1}}L f^{i-1}_j\|_2
+ \|\langle\nabla f^{i-1}_j, \nabla a^{j,i}_{f^{i-1}}\rangle\|_2\Bigr)
\\
+C(p, k ,d)\|f-g\|_2\|a^{j,i}_{f^{i-1}} \langle\nabla f^{i-1}_j,
\nabla \Delta_f \rangle\|_2\|(\Delta_f + \varepsilon)^{-2}\|_p
\\
\le C(k,d)\varepsilon^{-1} \Bigl
(\sum_{i=1}^k (\sigma^2_{f_i} + \sigma_{g_i}^2)\Bigr)^{k-1/2}\|f-g\|_2
\\
+C_1(p,k,d)\|f-g\|_2\Bigl(\sum_{i=1}^k (\sigma^2_{f_i}
+ \sigma_{g_i}^2)\Bigr)^{2k-1/2}\varepsilon^{-2+\beta/p}
\biggl(\int \Delta_f\, d\gamma\biggr)^{-\beta/p}.
\end{multline*}
Now the left-hand side of (\ref{f2}) can be estimated by
\begin{multline*}
 C_2(p,k,d)\biggl(\varepsilon^\beta
\biggl(\int \Delta_f\, d\gamma\biggr)^{-\beta}
+\Bigl(\sum_{i=1}^k (\sigma^2_{f_i} + \sigma_{g_i}^2)\Bigr)^{k-1/2}
\|f-g\|_2\varepsilon^{-1}
 \\
+\|f-g\|_2\Bigl(\sum_{i=1}^k (\sigma^2_{f_i} + \sigma_{g_i}^2)
\Bigr)^{2k-1/2}\varepsilon^{-2+\beta/p}
\biggl(\int \Delta_f\, d\gamma\biggr)^{-\beta/p}\biggr).
\end{multline*}
If $\|f-g\|_2\ge1$, the desired estimate (\ref{est1}) is trivial.
Assume that $\|f-g\|_2\le1$. Whenever $\varepsilon\le 1$ we have
$\varepsilon^{-1} \le\varepsilon^{-2+\beta/p}$.
Let $\tau = (p-1)/p$.
Setting $\varepsilon = \|f-g\|^\alpha$ with $\alpha = (2+ \beta\tau)^{-1}$, we have
$$
\|\gamma\circ f^{-1} - \gamma\circ g^{-1}\|_{\rm TV}
\le C_2(p,k,d)R(f,g)\|f-g\|^{\theta}, \ \theta=\frac{1}{4k(d-1)+\tau},
$$
 where
\begin{multline*}
R(f,g) = \biggl(\int \Delta_f\, d\gamma\biggr)^{-\beta} +
\Bigl(\sum_{i=1}^k (\sigma^2_{f_i} + \sigma_{g_i}^2)\Bigr)^{k-1/2}
\\
 +
\Bigl(\sum_{i=1}^k (\sigma^2_{f_i} + \sigma_{g_i}^2)\Bigr)^{2k-1/2}
\biggl(\int \Delta_f\, d\gamma\biggr)^{-\beta/p}.
\end{multline*}
Since $|\sigma_{f_i} - \sigma_{g_i}|\le2\|f-g\|_2\le2$,
the desired estimate is proved.
\end{proof}

\begin{remark}
\rm
Theorem \ref{t4.3} yields an analog
of estimate (\ref{est1}) with the power of the $L^2$-norm equal to $1/(4k(d-1)+1+\tau)$.
Hence Theorem~\ref{t6.3} provides a better rate of convergence.
\end{remark}

\bibliographystyle{amsplain}

\end{document}